\documentclass{amsart}

\usepackage{exscale}
\usepackage{color}
\usepackage{changes}

\usepackage{blkarray}
\usepackage{hyperref}
\usepackage{bm}
\usepackage{latexsym}

\usepackage{amscd}
\usepackage{amsmath}
\usepackage{amssymb}
\usepackage{amsthm}

\usepackage{hhline}

\usepackage{stmaryrd}
\usepackage{float}
\usepackage{graphicx}
\usepackage{psfrag}
\usepackage{pst-all}

\usepackage{tikz}
\usetikzlibrary{arrows,automata}

\theoremstyle{plain}
\newtheorem{theorem}{Theorem}[section]
\newtheorem{corollary}[theorem]{Corollary}
\newtheorem{lemma}[theorem]{Lemma}

\newtheorem{conjecture}[theorem]{Conjecture}

\newtheorem{notation}[theorem]{Notation}
\theoremstyle{definition}
\newtheorem{definition}[theorem]{Definition}
\newtheorem{remark}[theorem]{Remark}
\newtheorem{example}[theorem]{Example}
\newtheorem{proposition}[theorem]{Proposition}

\newcommand{\Lbb}{\mathbb{L}}
\newcommand{\Z}{\mathbb{Z}}

\newcommand{\R}{\mathbb{R}}
\newcommand{\C}{\mathbb{C}}

\newcommand{\PP}{\mathbb{P}}

\newcommand{\F}{{\mathbb F}}

\newcommand{\CC}{{\mathcal C}}
\newcommand{\last}{{\mathsf{last}}}
\newcommand{\mini}{{\mathsf{min}}}
\newcommand{\diag}{{\mathsf{diag}}}

\newcommand{\A}{{\mathcal A}}

\newcommand{\St}{{\mathcal S}}

\newcommand{\pol}{\lhd}

\definecolor{deepblue}{cmyk}{0,0.83,1,0.70}
\definecolor{gray}{cmyk}{0,0,0,0.3}
\definecolor{rred}{cmyk}{0,1,1,0}
\definecolor{chairo}{cmyk}{0,0.83,1,0.70}
\definecolor{roypur}{cmyk}{0.75,0.90,0,0.1}
\definecolor{darkorc}{cmyk}{0.40,0.80,0.20,0}
\definecolor{oliv}{cmyk}{0.64,0.00,0.75,0.56}
\definecolor{azuro}{cmyk}{1,1,0,0.46}

\DeclareMathOperator{\Cone}{Cone}

\usepackage{color}
\usepackage{ifthen}

\newboolean{usecolor}
\setboolean{usecolor}{true}

\ifthenelse{\boolean{usecolor}}
{
  \definecolor{colore}{cmyk}{0,1,0.6,0}
  \definecolor{coloregen}{cmyk}{0.7,0,1,0}
  \definecolor{coloresimo}{cmyk}{1,0.6,0,0}
}
{
  \definecolor{colore}{cmyk}{0,0,0,1}
  \definecolor{coloregen}{cmyk}{0,0,0,1}
  \definecolor{coloresimo}{cmyk}{0,0,0,1}
}

\title{Homology graph of real arrangements and monodromy of Milnor Fiber}
\author{Pauline Bailet}
\address{Department of Mathematics \\
Bremen University,  Bremen, Germany} 
\email{pauline.bailet@uni-bremen.de} 
\author{Simona Settepanella}
\address{Department of Mathematics \\
Hokkaido University, Sapporo, Japan} 
\email{s.settepanella@math.sci.hokudai.ac.jp }
\thanks{Authors gratefully acknowledge F. Callegaro, G. Gaiffi and M. Salvetti organizers of the workshop \textit{Arrangements: topology, combinatorics and stability} for hosting them at Pisa Mathematics Department. They also thank M. Falk, A. Libgober, M. Yoshinaga and an anonymus referee for their useful comments. 
During the preparation of this paper, first author was supported by two Postdoctoral Fellowships for Foreign Researchers: JSPS and then University of Bremen and the European Union FP7 COFUND under grant agreement n\textsuperscript{o} 600411). She thanks both of them for their financial and other supports. Second author is supported by JSPS Grant in aid n. 16K05051, title: \textit{On generalized Pure Braid Group.} }
\date{\today}

\begin{document}
\maketitle

\begin{abstract}

We study the first homology group $H_1(F,\C)$ of the Milnor fiber $F$ of sharp arrangements $\overline{\A}$ in $\PP^2_\R.$ 
Our work relies on the minimal complex $\CC_*(\mathcal{S}(\A))$ of the deconing arrangement $\A$ and its boundary map. We describe an algorithm which computes possible eigenvalues of the monodromy operator $h_1$ of $H_1(F,\C)$. We prove that, if a condition on some intersection points of lines in $\A$ 
is satisfied, then the only possible non trivial eigenvalues of $h_1$ are cubic roots of the unity. Moreover we give sufficient conditions for just eigenvalues of order $3$ or $4$ to appear in cases in which this condition is not satisfied.
\end{abstract}

\noindent

 \section{Introduction}

 Let $\overline{\A}$ in $\PP_{\R}^2$ be a projective line arrangement, $\A$ in $\R^2$ be its deconing and $M(\A)=\C^2 \setminus \cup_{H \in \A}H_\C$ be the complement of the complexified arrangement. The Milnor fiber $$F=Q^{-1}(1)\subset \C^3$$ of $\A$ is the smooth affine hypersurface defined as preimage of $1$ by the defining polynomial $Q$ of $\overline{\A} $. Consider the geometric monodromy action on $F,$ given by the multiplication by $\lambda= \exp(2i\pi / n+1),$ where $n+1$ is the cardinality of $\overline{\A}$.  This automorphism induces the monodromy operators in homology $$h_q: H_q(F,\C) \to H_q(F,\C).$$   It is known that we have the following equivariant decomposition 
\begin{equation}\label{eq:decomposition}
H_q(F,\C)=\bigoplus_{d \mid  n+1}[\C[t,t^{-1}]/\varphi_d]^{\beta_{q,d}}
\end{equation}
where each $\beta_{q,d}$ is the multiplicity of an eigenvalue of $h_q$ with order $d,$ and $\varphi_d$ is the cyclotomic polynomial of degree $d.$ 
The computation of the eigenspaces of the monodromy operators, i.e. the cyclic modules $[\C[t,t^{-1}]/\varphi_d]^{\beta_{q,d}}$ appearing in (\ref{eq:decomposition}), is a difficult question which has been intensively studied the last decades and approached by different techniques such as nonresonant conditions for local systems (see for intance \cite{cdo,co-loc}), multinets (\cite{FalYuz,Torielli-Yoshi}), minimality of the complement(\cite{SS,y-lef,y-mini,y-cham}), graphs (\cite{Bailet},\cite{Sal-Ser}), and also mixed Hodge structure (\cite{BDS,BS,BS1,DNA,DL,DP,D3}). Many progress have been done for braid arrangements (\cite{S1}), graphic arrangements (\cite{PM}) and real line arrangements (\cite{Yoshi2,Yoshi3}). Notice that we can reduce to study $h_1,$ since the eigenspaces of $h_1$ determine the eigenspaces of $h_2$ in view of the formula of the Zeta function of the monodromy, see for instance \cite[Proposition 4.1.21]{ST}. Although we know that the monodromy operators and their eigenspaces are closely related with the multiplicities of the intersection points of the arrangement (see among others \cite{Bailet2,Dimcamonodromy,lib-mil,PS2,Yoshi2}), the role of these multiplicities in the determination of the eigenspaces is obscure and many questions are still open, even for $q=1.$
We mention here some of them that motivated the work presented in this paper:
\begin{enumerate}
\item it is yet unknown whether the (co)homology groups of the Milnor fiber are combinatorial; 
\item it is unknown whether the (co)homology groups of the Milnor fiber with $\Z$ coefficients are torsion free (but partial results as in \cite{Wil});
\item it is still yet not known wether a cyclic module $[\C[t,t^{-1}]/\varphi_d]$ appears or not in the decomposition of $H_q(F,\C)$ for a given $d$. In particular it is not known, in general, wether a non trivial monodromy $d$ appears or not;
\item even more specifically, it is not yet known whether the monodromy operator $h_1$ can have eigenvalues which are not roots of unity of order $3$ or $4$ (see \cite{PS2}).
\end{enumerate}
In case no non trivial eigenvalues appear, the arrangement $\A$ is said to be \textit{a-monodromic}. Several conjectures have been made on the (co)homology groups of the Milnor fiber. Among others the four following ones this paper focuses on:

\begin{conjecture}\cite{Sal-Ser}\label{conj:amon} Let $\Gamma(\A)$ be the graph defined by vertices $H \in \A$ and edges $(H,H')$ if and only if  $H \cap H'$ is a point of multiplicity two. If $\Gamma(\A)$ is connected, then $\A$ is a-monodromic.
\end{conjecture}

\begin{conjecture}\label{conj:mon}\cite{Yoshi2} If $\overline{\A}$ has a sharp pair of lines and $\A$ is not a-monodromic, then the eigenvalues of the monodromy operator $h_1$ are cubic roots of unity.
\end{conjecture}

\begin{conjecture}\cite[Conjecture 1.9]{PS2}
The only possible non trivial monodromies of order prime powers have order $3$ or $4.$
\end{conjecture}

\begin{conjecture}\cite{Yoshi2} If $\overline{\A}$ is a simplicial arrangement and $\A$ is not a-monodromic, then the eigenvalues of the monodromy operator $h_1$ are cubic roots of unity. \end{conjecture}

The purpose of this paper is farther investigate those four conjectures in case of sharp arrangements, that is arrangements $\overline{\A}$ for which it exists a \textit{sharp} pair $(\overline{H},\overline{H'})$ of hyperplanes that satisfies the condition that all intersection points lie on $\overline{H},\overline{H'}$ or in the same region of $\PP_{\R}^2$ delimited by $\overline{H}$ and $\overline{H'}$.

In particular, in case of sharp arrangements, Conjecture \ref{conj:mon} is direct consequence of a more general conjecture which states that $d$-monodromy appears if and only if $d$-multinets appear.
Indeed by \cite[Theorem 3.1 (i)]{DP} and \cite[Theorem 3.11]{FalYuz} it is known that if a $d$-multinet on $\overline{\A}$ appears for some $d \geq 3$, then $\beta_{1,d} \geq d-2$. While by \cite[Theorem 3.21]{Yoshi2}, if $d \neq 1$ then $\beta_{1,d} \leq 1$. That is on sharp arrangements at most $3$ multinets can exist.

In this paper we study $H_1(F,\C)$ by mean of a minimal complex $\CC_*(\mathcal{S}(\A))$ homotopically equivalent to the complement $M(\A), $ builded in \cite{SS} by M. Salvetti and the second author. We give an alternative algorithm to compute $H_1(F,\C)$ to the one given in \cite{Yoshi2} by M. Yoshinaga and we prove that if a certain condition on some intersection points of lines in $\A$ is satisfied, then the only possible non trivial monodromy for the fiber $F$ has order $d=3$. Moreover if this condition is not satisfied, then we give sufficient conditions for monodromy to have order $d=3$ or $4$. \\
More in detail, if $\overline{\A}$ is a sharp arrangement and $\A$ is the affine arrangement obtained by removing one of the two lines in the sharp pair $(\overline{H},\overline{H'})$ ( e.g. $\overline{H}=H_{\infty}$ line at infinity),  then lines in $\A$ and intersection points in each line have a natural order, both denoted by $\pol$, given as follows. The other line in the sharp pair (e.g. $\overline{H'}$) is the line $H^{P_0}_1$, where $P_0=\overline{H} \cap \overline{H'}$. Fix a direction along line $H^{P_0}_1$. The intersection points in $H^{P_0}_1$ are ordered going along $H^{P_0}_1$ in positive direction and, for each point $P_i$, we order lines through it in the order they are intersected by a line $H^{P_0}_{(1,\epsilon)}$ oriented as $H^{P_0}_1$ and obtained translating $H^{P_0}_1$ by an epsilon in the direction of the half plane with no intersection point (see Figure \ref{fig1}). Finally, in each line, intersection points are ordered naturally from the first point which is the intersection with $H^{P_0}_1 (= \overline{H'})$ to the last real point  which is the last real intersection before the point at infinity (intersection with $\overline{H}$). Lines passing trough $P_0$ are oriented as $H^{P_0}_1$ and numbered from $H^{P_0}_2$, the closest one to  $H^{P_0}_1$, up to $H^{P_0}_{m(P_0)-1}$, last real one (here $m(P_0)$ is the multiplicity of $P_0$). With such an order, if $\last(H)$ is the last real intersection point along the line $H$, $\mini(H)$ the second one, and $H_j^{P_i}$ the $j$-th line passing trough the point $P_i \in H^{P_0}_1$, then $H_j^{P_0} \pol H_{j'}^{P_{0}}$ iff $j >j'$, $H_j^{P_i} \pol H_{j'}^{P_{i'}}$ iff $(i,j) <(i',j')$ in lexicographic order, and  the following results hold.
\begin{theorem}\label{thmint1} Let $\A \subset \R^2$ be a sharp arrangement such that $$\last(H_{m(P_i)-1}^{P_i}) \neq \last(H_2^{P_j}), \quad 0 \leq i <j, $$
holds for hyperplanes in $\A_{(0,3)}$. Then $\A$ is a- or 3-monodromic.
\end{theorem}
\begin{theorem}\label{thmint2}
 Let $\A$ in $\R^2$ be a sharp arrangement. If $\mathcal{G}_{\A'_{(0,3,4)},\last,\mini}$ does not contain any cycle of length $l$ as the one in Figure \ref{fig:cyclastmin} such that:
\begin{itemize}
\item $\overline{H}\rhd H'$ and $l$ is odd
\item $\overline{H}\pol H'$ and $l$ is even
\end{itemize} 
then $\A$ is a-, 3- or 4-monodromic.
\end{theorem}
Here $\A_{(0,3)}$ and $\A_{(0,3,4)}$ are the sub-arrangements of $\A$ defined, respectively, in equations (\ref{eq:A03}) and (\ref{eq:A034}), $\mathcal{G}_{\A'_{(0,3,4)},\last,\mini}$ is the graph defined in \ref{def:grafGLM} and vertices and edges in $\mathcal{G}_{\A'_{(0,3,4)},\last,\mini}$ and in cycle in Figure \ref{fig:cyclastmin} are drawn as described in Figure \ref{fig:edge}. \\
The above theorems partially answer to Conjecture \ref{conj:amon} as it has been proved in \cite{Bailet} by the first author that if the non trivial monodromies of the fiber $F$ have order a power of a prime, then Conjecture  \ref{conj:amon} holds. Moreover the algorithm described to prove theorems above can be applied, more in general, to study a-monodromicity of $\A$.\\
The paper is organized as follows. In Section \ref{sec:Mincomp} we recall the building foundation which are used throughout this paper. In Section \ref{maingraph} we describe a new graph based on the boundary map defined in \cite{gaiffi2009morse} and we give fundamental Lemma's that state relation between cycles in this graph and Gauss operations to triangularize the boundary map. In Section \ref{sec:sharp} we introduce main simplification that holds for boundary map when the arrangement $\A$ is sharp and, finally, in Section \ref{sec:mains} we define special graphs associated to the graph defined in Section \ref{maingraph} and we use them to prove Theorem \ref{thmint1} and Theorem \ref{thmint2}, stated as Theorem \ref{thm1.1} and Theorem \ref{thm2.1}. In the last Section we give examples to illustrate how our algorithm can be applied also to immediately state a-monodromicity.\\
Many questions are left open as, for instance, if there are conditions that allow to further simplify our algorithm to finally get Conjecture \ref{conj:mon} in case of sharp arrangements and, of course, what can be said if we extend it also to non sharp arrangements. This will be object of further studies.

\section{Minimal complex for real line arrangements}\label{sec:Mincomp}

In this section we will recall the complex defined in \cite{SS} in the special case of line arrangements (see also \cite{gaiffi2009morse}) and we will also introduce notations and definitions useful in the rest of the paper. 

Let us denote by $L(\A)$ the lattice of intersection of $\A$ and by $\St(\A)$ the stratification of $\R^2$ into facets induced by the arrangement.

\subsection{Minimal complex and boundary map}
 
 In \cite{SS} Salvetti and the second author described a minimal complex $\CC_*(\A)$ homotopically equivalent to the complement $M(\A)$ of a real complexified arrangement $\A.$ Their construction is based on three main ingredients:  
\begin{enumerate} 
\item a polar order on the stratification $\St(\A)$ induced by an \textit{ $\A$-generic} polar system of coordinates; 
\item the Salvetti's complex introduced in \cite{Sal};
\item the discrete Morse theory. 
\end{enumerate}
In order to fix a polar system of coordinates in $\R^2$, an origin $V_0$ and a line $V_1$ containing $V_0$ are needed; then the polar coordinates of a point will be determined by the distance from $V_0$ and the rotation angle of the line $V_1$ inside $\R^2$. 
In order for this polar system to be \textit{$\A$-generic}, $V_0$ has to lie in an unbounded chamber $C_0$ of $\St(\A)$ and the line $V_1$ containing $V_0$ has to intersect all hyperplanes $H_i \in \A$ once in an unbounded facet $F_i \subset H_i$ of the stratification $\St(\A)$. Let us remark that, with this choice of polar system, all the points $P \in L(\A)$ lie on the same side of $V_1$. \\
 We will denote by $\pol$ the polar order induced on the stratification $\St(\A)$ of $\R^2$ by the polar system of coordinates $(V_0$, $V_1)$. This restricts to an order on the hyperplanes of $\A$ given by the distance from $V_0$ of their intersection with $V_1$.\\
Let us now recall that the $k$-cells of the Salvetti's complex are pairs $[C \prec F^k]$ in which $C$ and $F^k$ are, respectively, a chamber and a $k-$facet of the stratification $\St(\A)$ such that $F^k$ is contained in the closure of $C$. More in general if $F,G \in \St(\A)$ we write $G \prec F$ if $F$ is contained in the closure of $G$.\\ 
The cells of the minimal complex are the cells of the Salvetti's complex which are critical with respect to a suitable discrete Morse function. Explicitely in case of line arrangement:

\begin{itemize}
\item the critical $2-$cells are the pairs $[C \prec P]$ such that the point P is maximal with respect to the polar ordering, among the facets of $\St(\A)$ contained in the closure of $C$;
\item the critical $1-$cells are the pairs $[C \prec F]$ such that $C \pol F$ and $F$ is a $1$-dimensional facet intersected by $V_1$;
\item the only critical $0-$cell is the pair $[C_0\prec C_0]$, $V_0 \in C_0$.
\end{itemize}

In \cite{gaiffi2009morse} Gaiffi and Salvetti gave, in case of line arrangements, a simplified version of the boundary map defined in \cite{SS}. Following their notations, for any point $P\in L_2(\A),$ denote by $$S(P):=\{H\in \A\,|\,P \in H\}$$ the set of hyperplanes containing $P.$ For a given critical $2$-cell  $[C \prec P],$ there are two lines in $S(P)$ which bound $C$. Denote by $H_C$ (resp. $H^C$) the one which has minimum (resp. maximum) index. Let $\Cone(P)$ be the closed cone bounded by the minimum and the maximum lines of $S(P),$ having vertex $P$ and whose intersection with $V_1$ is bounded. Then define
\begin{equation*}
\begin{split}
& U(C):=\{H_i \in S(P) \mid i \geq \mbox{ index of } H^C \},\\
& L(C):=\{H_i \in S(P) \mid i \leq \mbox{ index of } H_C \}.\\
\end{split}
\end{equation*}

Consider a line $V_{P}$ that passes through the points $V_0$ and $P$. This line intersects all the lines $H\in \A$  in points $P_H$ and, if $\theta(P_H)$ is the length of the segment $V_0P_H$, then define the set 
$$U(P):=\{H \in \A \mid \theta(P_H)> \theta(P)\}.$$
Then the map $\partial_2$ in \cite{gaiffi2009morse} is given by

\begin{equation}\label{eq:boundary2}
\begin{split}
&\partial_2(l.e_{[C \prec P]})=\sum_{\substack{ |F_i|\in S(P)}}\bigg(\prod_{\substack{k<i~s.t.\\H_k\in U(P)}}t_{H_k}\bigg)\bigg[\prod_{\substack{k~s.t.\\H_k\in[C\rightarrow|F_i|)}}t_{H_k}-\prod_{\substack{k<i~s.t.\\H_k\in S(P)}}t_{H_k}\bigg](l).e_{[C_{i-1}\prec F_i]}+\\
&\sum_{\substack{|F_i|\in U(P)\\F_i\subset \Cone(P)}}\bigg(\prod_{\substack{k<i~s.t.\\H_k\in U(P)}}t_{H_k}\bigg)\bigg(1-\prod_{\substack{k<i~s.t.\\ H_k\in L(C)}}t_{H_k}\bigg)\bigg(\prod_{\substack{k<i~s.t.\\H_k\in U(C)}}t_{H_k}-\prod_{\substack{k~s.t.\\H_k\in U(C)}}t_{H_k}\bigg)(l).e_{[C_{i-1}\prec F_i]},
\end{split}
\end{equation}
where $l \in \Lbb$ is an abelian local system, $|F_i|=H_i$ is the affine subspace spanned by $F_i,$ and $[C\rightarrow |F_i|)$ are the subsets of $S(P)$ defined by
\begin{enumerate}
\item[i)] $[C\rightarrow|F_i|):=\{H_k\in U(C)~|~k<i\}$ if $|F_i|\in U(C)$;
\item[ii)] $[C\rightarrow|F_i|):=\{H_k\in S(P)~|~k<i\}\cup U(C)$ if $|F_i|\in L(C)$.
\end{enumerate}
Furthermore, because the only critical $0$-cell is $[C_0\prec C_0]$, the boundary map $\partial_1$ can be easily computed:
$$\partial_1(l.e_{[C_{i-1}\prec F_i]})=(1-t_{H_i})(l).e_{[C_0\prec C_0]}.$$

Notice that computing the (co)homology of the Milnor fiber with integer coefficients is equivalent to set, in the above boundary map, all the elements $t_H \in$ Aut$(\Lbb)$ equal to the same $t\in$ Aut$(\Lbb)$.

\begin{remark}
Notice that each entry of the boundary matrix given by formula (\ref{eq:boundary2}) above is divisible by $1-t,$ by minimality of the complex $\CC_*(\A).$
\end{remark}

\subsection{Polar ordering and indexation of the lattice $L(\A)$.}\label{polarorder}
In this subsection we will list definitions and notations useful in the rest of the paper. Let $L_2(\A)$ be the set of rank $2$ elements in the intersection lattice $L(\A)$, i.e. intersection points of lines in $\A$.
\begin{enumerate}
\item[1.] It is an obvious remark that the critical $1$-cells $[C \prec F]$ are in one to one correspondence with the hyperplanes $H= |F|$ of the arrangement. Hence in the rest of the paper we will simply use $H$ instead of $[C \prec F]$.
\item[2.] The second boundary map defined in equation (\ref{eq:boundary2}) is the sum of two addends, one is a summation on the lines belonging to $S(P)$, the second one a summation on the lines belonging to the set $U(P)$ which are in the Cone of $P$. Given a point $P \in L_2(\A)$ let's then define the \textit{Upper Cone} of $P$ as the set of lines
\begin{equation}\label{eq:upper}
\bold{\widehat{U}}(P):=\{ H=|F| \in U(P) \mid F\in \Cone(P)\}.
\end{equation}
\begin{example}Looking at the arrangement in Figure \ref{fig1}, $\Cone(P)$ is the Cone between the two lines $H_{m(P_1)}^{P_1}$ and $H_{2}^{P_m}$. Any line $H^{P_i}_j$, $1<i<m$, is in  $\Cone(P)$, but, for example, the line $H_{2}^{P_2}$ is in the upper $U(P)$ of P, i.e. $H_{2}^{P_2} \in \bold{\widehat{U}}(P)$, while $H_{m(P_2)}^{P_2}$ passes clearly below the point $P$, i.e. $H_{m(P_2)}^{P_2} \notin \bold{\widehat{U}}(P)$.
\end{example}
\item[3.] Given a point $P\in L_2(\A),$ we will denote the hyperplanes in $S(P)$ by 
\begin{equation}\label{notationlocal}
H_1^P \pol \cdots \pol H_{m(P)}^P,
\end{equation}
where $m(P)$ is the multiplicity of $P$ in $\A$ and the hyperplanes $H_i^P$ are ordered following the order $\pol$ in $\A.$
\item[4.] Given a point $P\in L_2(\A)$ and following the previous notation for the lines in $S(P),$ we denote by $C_j^P$ the unique chamber  $C_j^P \pol P$ with walls facets $F$ and $F'$, $|F|=H_{j}^P, \,|F'|=H_{j+1}^P$.
\item[5.] Given a line $H$, the order $\pol$ on the stratification $\St(\A)$ induces a local order on the intersections points 
\begin{equation}\label{ord:points}
P_1^H \pol  P_2^H  \pol \cdots \pol P_{k-1}^H \pol P_k^H
\end{equation}
in $L_2(\A)$ belonging to $H$. In particular in the rest of the paper we will denote by $\last(H)$ the last point in the above order and by $\mini(H)$ the second one. (This latter notation will be clearer in the following of the paper).
\end{enumerate}

\begin{remark} The description of the set $U(P)$ (and hence $\bold{\widehat{U}}(P)$) can be given in terms of the intersection lattice simply using the order in (\ref{ord:points}). Indeed, for example, if $P=P_i^H$ is a point in $H$, a line $H' \pol H$ will be in $U(P)$ if and only if $H' \cap H \neq \emptyset$ is a point $P_j^H$ with $j < i$.  
\end{remark}

\subsection{The Milnor Fibre of line arrangements}

Let $Q: \C^3 \rightarrow \C$ be the defining polynomial (of degree $n+1$) of the projective line arrangement $\overline{\A}$. Then the Milnor Fiber is defined as $F=Q^{-1}(1)$ with geometric monodromy
\begin{equation*}
\begin{split}
\pi_1(\C^*,1&) \rightarrow Aut(F) \\
& \alpha \rightarrow e^{\frac{2\pi i}{n+1}}\alpha 
\end{split}
\end{equation*}
that induces the monodromy operator $h_q: H_q(F,A) \to H_q(F,A)$, A being any unitary commutative ring. Let $R:=A[t,t^{-1}]$ and $R_t$ be the ring $R$ endowed with the $\pi_1(M(\overline{\A}))$-module structure given by the abelian representation
$$
\pi_1(M(\overline{\A})) \rightarrow H_1(M(\overline{\A}); A) \rightarrow Aut(R)
$$ 
taking a generator $\beta_j$ into $t$-multiplication. Then it is a well know fact that
$$
H_*(M(\overline{\A}); R_t) \simeq H_*(F;A)
$$
where $t$-multiplication on the left corresponds to the monodromy action on the right. \\
Consider now $A=\C$, $R=\C[t,t^{-1}]$, then $H_*(M(\overline{\A}); R_t)$ decomposes into cyclic modules either isomorphic to $R$ or to $\frac{R}{(\varphi_d)}$, where $\varphi_d$ is the $d$-cyclotomic polynomial with $d \mid (n+1)$. Following notations in \cite{Sal-Ser} we will call an arrangement $\overline{\A}$ \textit{a-monodromic} if it has trivial monodromy, that is if and only if 
$$
H_1(F;\C) \simeq \C^n \quad 
$$
and, more in general, \textit{d-monodromic} if the cyclic module $\frac{R}{(\varphi_d)}$ appears in the decomposition of $H_1(F;\C)$. Analogously if $ \A$ is the affine arrangement deconing of $\overline{\A}$, we will call $\A$ \textit{a-monodromic} if 
$$
H_1(M(\A);R_t) \simeq \C^n \quad 
$$
and \textit{d-monodromic} if the cyclic module $\frac{R}{(\varphi_d)}$ appears in the decomposition of $H_1(M(\A);R_t)$.
Remark that if $\A$ is a-monodromic $\overline{\A}$ is (the converse is not true in general).

\section{Homology graph of line arrangements}\label{maingraph}

Let $Mat(\partial_2)$ be the second boundary matrix defined by the map in equation (\ref{eq:boundary2}). For the sake of simplicity, we will denote by 
\begin{enumerate}
\item[i)] $H$ the row corresponding to the critical $1-$cell $[C \prec F],\,|F|=H$; 
\item[ii)] $c^P$ the columns block relative to the critical $2-$cells $[C \prec P]$.
\end{enumerate}
More in detail, following the local notations in subsection \ref{polarorder}, we will denote by $c_{j}^P$ the column corresponding to the critical $2-$cell $[C_{j}^P\prec P].$ With these notations we have $$c^P=\{c_1^P,\hdots,c_{m(P)-1}^P\}.$$
Rows and columns of $Mat(\partial_2)$ are ordered following the polar ordering of the corresponding critical $1$ and $2-$cells. We have the following straightforward remark.

\begin{notation}\label{not:goodcolumn} Given a point $P \in L_2(\A)$ and a line $H \in S(P)$, if $e(H,c)$ denotes the entry of $Mat(\partial_2)$ corresponding to row $H$ and column $c$, then 
$$e(H_1^P,c_{m(P)-1}^P)=t^{\alpha_1}(1-t)$$ and $$e(H_i^P,c_1^P)=t^{\alpha_i}(1-t), \quad i>1, $$
where $\alpha_1$ and the $\alpha_i$ are positive integers.
Hence we denote by $c_H^P$ the column of the block $c^P$ such that
\begin{equation*}
\begin{split}
& c_H^P = c_{m(P)-1}^P, \,\,\mbox{if} \,\, H=H_1^P,\\
& c_H^P= c_1^P\,\mbox{otherwise},
\end{split}
\end{equation*} 
in such a way that $e(H,c_H^P)$ has the form $t^\alpha(1-t).$
\end{notation}

Our goal is to diagonalize the matrix $Mat(\partial_2).$ In order to do it we introduce the \textit{homology graph} $\mathcal{G}(\A)$ defined as follows: 
\begin{enumerate}
\item[i)] vertices are couples $(H,P),$ with $H\in \A$ and $P\in L_2(\A)$ such that $P\in H$;
\item[ii)] two vertices $(H,P)$ and $(H',P^\prime)$ are connected by an edge $[(H,P),(H',P'),\pol]$ (resp. $[(H,P),(H',P'),\rhd]$) oriented from $(H,P)$ to $(H',P^\prime)$ if and only if $H^\prime \in S(P)\cup \bold{\widehat{U}}(P)$ (see Figure \ref{fig:edge}) and $H \pol H^\prime$ (resp. $H \rhd H^\prime$).
\end{enumerate}
For simplicity, we will sometimes denote by $[(H,P),(H',P')]$ the edge from $(H,P)$ to $(H',P^\prime)$ when the order between $H$ and $H'$ is not needed.
\begin{figure}[htbp]
\centering
\begin{tikzpicture}

\tikzset{vertex/.style = {shape=rectangle,minimum size=0.1em}}
\tikzset{edge/.style = {->,> = latex'}}
\node[vertex] (a) at  (0,0) {$(H,P)$};
\node[vertex] (b) at  (3,0) {$(H',P')$};

\node[vertex] (c) at  (1.5,0.4) {$\pol$};

\draw[edge] (a.east)  to[bend left] (b.west);

\end{tikzpicture}
\caption{Edge $[(H,P),(H',P'),\pol]$ in $\mathcal{G}(\A)$: $H,H' \in \A$, $H \pol H'$, $P,P'$ intersection points $P \in H,P' \in H'$ and $H^\prime \in S(P)\cup \bold{\widehat{U}}(P)$, $S(P)$ set of lines containing $P$ and $\bold{\widehat{U}}(P)$ defined in (\ref{eq:upper}).}
\label{fig:edge}
\end{figure}

Given the matrix $Mat(\partial_2),$ the graph $\mathcal{G}(\A)$ defines operations on $Mat(\partial_2)$ as follows. The edge $[(H,P),(H^\prime,P^\prime)]$ in $\mathcal{G}(\A)$ is equivalent to the operation 
\begin{equation}\label{mainop}
O_{H,H^\prime}^{P,P^\prime}(\varphi^\prime)= H^\prime - \varphi^\prime H
 \end{equation}
in $Mat(\partial_2),$ where $\varphi^\prime \in \C[t,t^{-1}]$ is the polynomial satisfying $e(H^\prime, c_H^P)= \varphi^\prime e(H,c_H^P),$ $c_H^P$ being the column defined in Notation \ref{not:goodcolumn}. Notice that $\varphi^\prime$ exists as $e(H,c_H^P)=t^\alpha(1-t).$ We will say that we perform operations along a subgraph in $\mathcal{G}(\A)$ if we perform all the rows operations $O_{H,H^\prime}^{P,P^\prime}$ for edges $[(H,P),(H^\prime,P^\prime)]$ in this subgraph. The following definition will be used often in the rest of the paper.

\begin{definition}\label{def:remove}
We say that an hyperplane $H$ is \textit{removed} by operations along a subgraph $\delta$ in $\mathcal{G}(\A)$ if $(H,P)$ is a vertex in $\delta$ and performing operations $O_{H,H^\prime}^{P,P^\prime}, H^\prime \rhd H$, along this subgraph, the column $c_H^P$ is reduced to a column with all entries $e(H^\prime, c_H^P)=0$ if $H^\prime \rhd H$ and $e(H,c_H^P)=1-t.$ All other entries of $c_H^P$ are left unchanged.
\end{definition}

\begin{remark}\label{rk:removetrivial}
Let us remark that if $H$ is removed by operations along a subgraph $\delta$, then using simple columns operations, the row $H$ can be reduced to a row with all entries 0 but $e(H,c_H^P)=1-t$ without changing the entries $e(H^\prime,c^{P^\prime}),$ for all the $H^\prime \rhd H.$  
Hence if we can find a subgraph of $\mathcal{G}(\A)$ that allows us to remove all hyperplanes of $\A$ at once,
then the matrix $Mat(\partial_2)$ can be triangularized with pivots $1-t$ and rows ordered following the polar order, that is $\A$ is a-monodromic. 
\end{remark}

From now on we will assume that operations along the graph are performed following the polar order, that is performing all operations $O_{H,H^\prime}^{P,P^\prime}$ for all $H^\prime \rhd H$, where each $H$ is fixed step by step following the polar order.
In order to find a suitable graph $\delta$ to remove all hyperplanes at once, we need to study how the entry $e(H,c_H^P)= t^\alpha(1-t)$ is changed when we perform row operations along $\delta$ to remove $\overline{H}\pol H.$ It is an easy remark on Gauss reduction that we can keep removing hyperplanes $\overline{H}\pol H$ along the graph $\delta$ following the row (i.e. polar) order, without affecting entry $e(H,c_H^P)= t^\alpha(1-t)$ until we get the following submatrix (where the symbol * means a non zero entry) 
\begin{equation}\label{submatrixproblem}
\bordermatrix{
&c_{H'}^{P'}&c_{H}^{P}\cr
H'&S(P') & *\cr
H& * & S(P)\cr}
\end{equation}
for a given row $H^\prime \pol H$ such that $(H',P') \in \delta$. Indeed, in this case, in order to simplify the entry $e(H,c_{H^\prime}^{P^\prime})$ in the column $c_{H^\prime}^{P^\prime},$ we perform the row operation $H- \varphi H',$ where $\varphi\,\,\text{satisfies}\,\,e(H,c_{H'}^{P'})= \varphi e(H',c_{H'}^{P'}).$ Then the entry $e(H,c_H^P)=t^\alpha (1-t)$ is changed and, in general, we don't know if it will keep the form $t^\beta (1-t).$ 

\begin{lemma}\label{lemmacycle}If performing row operations along a subgraph $\delta$ of $\mathcal{G}(\A)$ to triangularize the matrix $Mat(\partial_2),$ we get the submatrix (\ref{submatrixproblem}) in the Gauss transformed matrix, then there exists a cycle in $\delta$ as the one in Figure \ref{fig:cycle} that contains $(H,P)$ and $(H^\prime,P^\prime)$. 

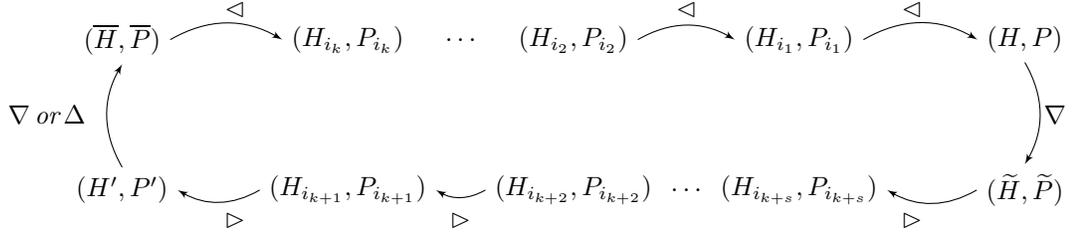
\begin{figure}[htbp]
\centering
\begin{tikzpicture}

\tikzset{vertex/.style = {shape=rectangle}}
\tikzset{edge/.style = {->,> = latex'}}
\node[vertex] (h) at  (12,0) {$(H,P)$};
\node[vertex] (h-h1) at  (10.5,0.4) {$\pol$};
\node[vertex] (h1) at  (9,0) {$(H_{i_1},P_{i_1})$};
\node[vertex] (h1-h2) at  (7.5,0.4) {$\pol$};
\node[vertex] (h2) at  (6,0) {$(H_{i_2},P_{i_2})$};

\node[vertex] (pointshaut) at  (4.5,0) {$\hdots$};

\node[vertex] (hk) at  (3,0) {$(H_{i_k},P_{i_k})$};\node[vertex] (hk-hbar) at  (1.5,0.4) {$\pol$};
\node[vertex] (hbar) at (0,0) {$(\overline{H},\overline{P})$};

\node[vertex] (h'-hbar) at  (-1,-1) {$\nabla \,\text{or}\,\Delta$};
\node[vertex] (h') at  (0,-2) {$(H',P')$};
\node[vertex] (h'-hk1) at  (1.5,-2.4) {$\rhd$};
\node[vertex] (hk1) at  (3,-2) {$(H_{i_{k+1}},P_{i_{k+1}})$};
\node[vertex] (hk1-hk2) at  (4.5,-2.4) {$\rhd$};
\node[vertex] (hk2) at  (6,-2) {$(H_{i_{k+2}},P_{i_{k+2}})$};

\node[vertex] (pointsbas) at  (7.5,-2) {$\hdots$};

\node[vertex] (hks) at  (9,-2) {$(H_{i_{k+s}},P_{i_{k+s}})$};
\node[vertex] (hks-htilde) at  (10.5,-2.4) {$\rhd$};
\node[vertex] (htilde) at  (12,-2) {$(\widetilde{H},\widetilde{P})$};

\node[vertex] (h-htilde) at  (12.4,-1) {$\nabla$};


\draw[edge] (hbar.east)  to[bend left] (hk.west);
\draw[edge] (h2.east)  to[bend left] (h1.west);
\draw[edge] (h1.east)  to[bend left] (h.west);
\draw[edge] (h'.north)  to[bend left] (hbar.south);
\draw[edge] (hk1.west)  to[bend left] (h'.east);
\draw[edge] (hk2.west)  to[bend left] (hk1.east);
\draw[edge] (htilde.west)  to[bend left] (hks.east);
\draw[edge] (h.south)  to[bend left] (htilde.north);

\end{tikzpicture}
\caption{Cycle $\gamma$ in $\delta$}
\label{fig:cycle}
\end{figure}
\end{lemma}

\proof
Since we get the submatrix (\ref{submatrixproblem}) it means that $e(H,c_{H'}^{P^{\prime}})\neq 0$ (a) or $\rightsquigarrow \neq 0$ (b) and $e(H',c_H^{P})\neq 0$ (c) or $\rightsquigarrow \neq 0$ (d), where the symbol $\rightsquigarrow \neq 0$ means that the entry became non zero after having performed previous rows operations on $Mat(\partial_2)$. Let us study these cases separately.
\begin{enumerate}
\item[(a)] corresponds to the existence of the edge $[(H',P'),(H,P),\pol],$ i.e. $H_{i_1}=\cdots = H_{i_k} = \overline{H} = H'$ in Figure \ref{fig:cycle};
\item[(b)] corresponds to existence of $\overline{H} \pol H$ with $e(\overline{H},c_{H'}^{P^\prime})\neq 0$ and a chain of operations with starting vertex $\overline{H}$ and ending vertex $H$ corresponding to the edges 
$$[(\overline{H},\overline{P}),(H_{i_k},P_{i_k}),\pol],\hdots,[(H_{i_2},P_{i_2}),(H_{i_1},P_{i_1}),\pol],[(H_{i_1},P_{i_1}),(H,P),\pol].$$
Note that we can have either $\overline{H} \pol H^\prime$ and the corresponding edge is $[(H',P'),(\overline{H},\overline{P}),\rhd],$ or $\overline{H} \rhd H^\prime$ and the corresponding edge is $[(H',P'),(\overline{H},\overline{P}),\pol];$
\item[(c)] corresponds to existence of the edge $[(H,P),(H',P'),\rhd],$ i.e. $H'= H_{i_{k+1}}=\hdots= H_{i_{k+s}}= \widetilde{H}$ in Figure \ref{fig:cycle};
\item[(d)] corresponds to existence $\widetilde{H} \pol H^\prime$ with $e(\widetilde{H},c_H^{P})\neq 0$ and a chain of operations with starting vertex $\widetilde{H}$ and ending vertex $H^\prime$ corresponding to the edges
$[(\widetilde{H},\widetilde{P}),(H_{i_{k+s}},P_{i_{k+s}}),\pol],\hdots,[(H_{i_{k+2}},P_{i_{k+2}}),(H_{i_{k+1}},P_{i_{k+1}}),\pol],[(H_{i_{k+1}},P_{i_{k+1}}),(H',P'),\pol].$

Note that in this case it is clear that $\widetilde{H} \pol H$ since $\widetilde{H} \pol H^\prime \pol H,$ and we have the edge $[(H,P),(\widetilde{H},\widetilde{P}),\rhd].$ 
\end{enumerate}
\endproof

\begin{remark}\label{rkamono} Condition in Lemma \ref{lemmacycle} is not sufficient in general to say that the matrix can not be triangularized. Indeed, it could exist another more suitable point $P'' \in H^\prime$ such that we have the submatrix 
$\bordermatrix{
& c_{H'}^{P''}&c_{H'}^{P'}&c_{H}^{P}\cr
H'& S(P'')&S(P') & *\cr
H & 0& * & S(P)\cr}$
and $c_{H^\prime}^{P''}$ can be simplified without changing $H,$ or it can happen that even if we perform $H-\varphi H^\prime$, still $e(H,c_H^P)$ is of the form $t^\beta(1-t).$ But anyway, Lemma \ref{lemmacycle} is sufficient to say that if no cycle appears in a suitable subgraph $\delta \subset \mathcal{G}(\A)$, then the submatrix $M^\prime$ whose rows $H$ and columns $c_H^P$ correspond to the vertices $(H,P)$ of $\delta$ can be triangularized. Then in particular, if we can find a suitable subgraph $\delta$ with as many vertices $(H,P)$ as the hyperplanes of $\A$ and such that for any two vertices $(H,P) \neq (H',P') \in \delta$, $H \neq H'$ and $P \neq P'$, then cycles as in Figure \ref{fig:cycle} of $\delta$ give informations on the monodromy of $\A.$ For example, if $\delta$ has no cycle it means that we have a submatrix $M^\prime$ with as many rows as $Mat(\partial_2)$  that can be reduced to $\diag(1-t),$ that is $\A$ is a-monodromic.
\end{remark}

In the following sections we will study the monodromy of the Milnor fiber of a special family of line arrangements, the so called sharp arrangements, studying the homology $H_1(M(\A); R_t)$ via suitable subgraphs of the graph $\mathcal{G}(\A)$.

\section{Sharp arrangements and their boundary map }\label{sec:sharp}

In the following of this paper we will deal with a special category of line arrangements, the so called \textit{sharp arrangements.}

\begin{definition}\label{sharppair}
A \textit{sharp pair} of a projective arrangement $\overline{\A}$ in $\PP^2_{\R}$ is a couple of lines $(\overline{H},\overline{H'})$ of $\overline{\A}$ such that all intersection points of lines in $\overline{\A}$ lie on $\overline{H} \cup \overline{H'}$ or are contained in one of the two regions of the projective plane $\PP^2_{\R}$ divided in two by $\overline{H}$ and $\overline{H'}.$
\end{definition}

\begin{definition}\label{sharparrangement}
We say that $\A$ in $\R^2$ is a \textit{sharp arrangement} if it is the deconing of an arrangement $\overline{\A}$ in $\PP_{\R}^2$ with respect to the hyperplane at infinity $H_\infty,$ where $H_\infty=\overline{H} \in \overline{\A}$ is a line in a sharp pair $(\overline{H},\overline{H'})$ of $\overline{\A}.$
\end{definition}
Let $\A$ be a sharp arrangement, $(\overline{H},\overline{H'})$ the sharp pair, $P_0= \overline{H} \cap \overline{H'}$ be the point at infinity, and $m(P_0)$ the multiplicity of $P_0$ in $\overline{\A}.$ Then keeping the notation (\ref{notationlocal}) we define $\overline{H}=H_\infty=H_{m(P_0)}^{P_0}$ and $$S(P_0)=\{H\in \A\,|\, \overline{H} \cap H_\infty = P_0\}.$$ We choose the pair $(V_0,V_1)$ as follows: $V_0$ lies in the unbounded chamber with wall the hyperplane $H'' \in S(P_0)$ such that the induced polar order in $S(P_0)$ is 
$$H''=H_{m(P_0)-1}^{P_0} \pol \cdots \pol H_{1}^{P_0}=H',$$ where $\overline{H'}$ is the second line in the sharp pair (see Figure \ref{fig1}).

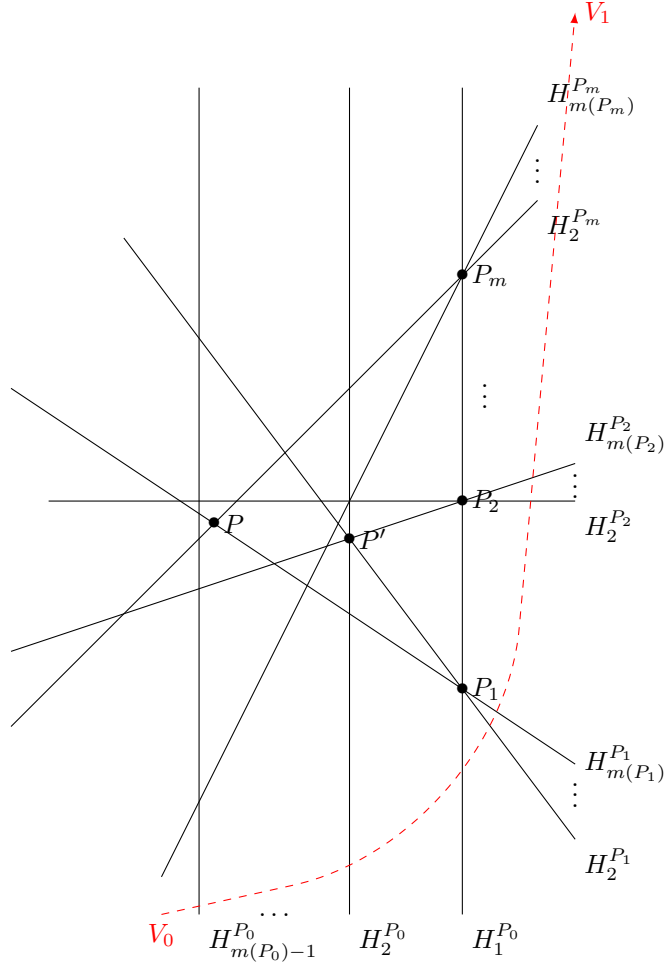
\begin{figure}[htbp]
\centering
\begin{tikzpicture}[scale=1]

\draw (0,0) node [below right] {$H_{m(P_0)-1}^{P_0}$} --(0,11);
\draw (2,0) node [below right] {$H_2^{P_0}$} --(2,11);
\draw (3.5,0) node [below right] {$H_1^{P_0}$} --(3.5,11);

\draw (5,1) node [below right] {$H_2^{P_1}$} --(-1,9);
\draw (5,2) node [right] {$H_{m(P_1)}^{P_1}$} --(-2.5,7);

\draw (5,5.5) node [below right] {$H_2^{P_2}$} --(-2,5.5);

\draw (5,6) node [above right] {$H_{m(P_2)}^{P_2}$} --(-2.5,3.5);

\draw (4.5,9.5) node [below right] {$H_2^{P_m}$} --(-2.5,2.5);

\draw (4.5,10.5) node [above right] {$H_{m(P_m)}^{P_m}$} --(-0.5,0.5);

\draw (3.5,3) node {$\bullet$}; 
\draw (3.5,3) node[right]{$P_1$} ;

\draw (3.5,5.5) node {$\bullet$}; 
\draw (3.5,5.5) node[right]{$P_2$} ;

\draw (3.5,8.5) node {$\bullet$}; 
\draw (3.5,8.5) node[right]{$P_m$} ;

\draw (2,5) node {$\bullet$}; 
\draw (2,5) node[right]{$P'$} ;

\draw (0.2,5.2) node {$\bullet$}; 
\draw (0.2,5.2) node[right]{$P$} ;

\draw (1,0) node {$\hdots$};

\draw (5,1.7) node {$\vdots$};

\draw (5,5.8) node {$\vdots$};

\draw (4.5,10) node {$\vdots$};

\draw (3.8,7) node {$\vdots$};

\draw [-latex, red, dashed, rounded corners=80pt] (-0.5,0) node [below] {$V_0$} --(4,1)--(5,12) node [right] {$V_1$};

\end{tikzpicture}
\caption{Polar coordinates system for a sharp arrangement}
\label{fig1}
\end{figure}

\begin{remark}
The choice of notation $P_0$ for the point at infinity will be clearer in the next section. Indeed it turns out to be the most natural choice to agree with the properties on all the other points in the sharp line $H_1^{P_0}$. 
\end{remark}

Denote by $$\mathcal{P}:=\{P \in L_2(\A)\,|\, P \in H_{1}^{P_0}\}$$
the set of all affine points belonging to the sharp line $H_{1}^{P_0}$. As $H_{1}^{P_0}  \cup {P_0}=\overline{H'}$ in the sharp pair $(\overline{H},\overline{H'})$ of $\overline{\A},$ it is an easy remark that for any point $P\in \mathcal{P},$ $\bold{\widehat{U}}(P)=\emptyset,$ as any other point $P'\in L_2(\A)\backslash \mathcal{P}$ has to lie in the same side of $H_{m(P_0)-1}^{P_0}$ with respect to $H_{1}^{P_0}$ (see Figure \ref{fig1}). As in Figure \ref{fig1}, we will index points $P_1 \pol \cdots \pol P_m$ in $\mathcal{P}$ following the order induced by $(V_0,V_1)$ along $H_{1}^{P_0}.$

\subsection{On the line $H_1^{P_0}$}\label{inclusion}
The arrangement $\A_1=\{H_1^{P_0}\}$ is a subarrangement of $\A$ and it is a simple remark that the complex $\CC_*(\mathcal{S}(\A_1))$ is a subcomplex of $\CC_*(\mathcal{S}(\A)).$ Then the map \begin{equation} \label{eq:inclusion1}
\begin{split}
i\colon (\CC_*(\mathcal{S}(\A_{1})), \partial_*) &\rightarrow  (\CC_*(\mathcal{S}(\A)),\partial_*)
\end{split}
\end{equation}
is a well defined inclusion of algebraic complexes. It defines an exact sequence of algebraic complexes

\begin{equation}\label{eq:seq1}
0 \rightarrow (\mathcal{C}_*(\mathcal{S}(\A_{1})),\partial_*) \rightarrow(\mathcal{C}_*(\mathcal{S}(\A)),\partial_*) \rightarrow (F^1_*(\mathcal{S}(\A)),\partial_*)\rightarrow 0,
\end{equation}
where $F^1_*(\mathcal{S}(\A))$ is the quotient complex $\CC_*(\mathcal{S}(\A))/\CC_*(\mathcal{S}(\A_{1}))$ with the  induced boundary map. This exact sequence (\ref{eq:seq1}) gives rise to a long exact sequence in homology
\begin{center}
$\dots \rightarrow H_2((F^1_*(\mathcal{S}(\A)),\partial_*)) \rightarrow H_1(\mathcal{C}_*(\mathcal{S}(\A_1)),\partial_*) \rightarrow H_1(\mathcal{C}_*(\mathcal{S}(\A)),\partial_*)$\\
$\rightarrow H_1(F^1_*(\mathcal{S}(\A)),\partial_*)\rightarrow H_0(\mathcal{C}_*(\mathcal{S}(\A_1)),\partial_*) \rightarrow \dots$
\end{center}
with
\begin{center}
$H_1(\mathcal{C}_*(\mathcal{S}(\A_{1})),\partial_*)=0$ and  $H_0(\mathcal{C}_*(\mathcal{S}(\A_{1})),\partial_*) =\C[t,t^{-1}]/1-t,$
\end{center}

that is we have the same cyclic modules $\frac{R}{\varphi_d}$ in $H_1(\mathcal{C}_*(\mathcal{S}(\A)),\partial_*)$ and $H_1(F^1_*(\mathcal{S}(\A)),\partial_*)$. It is an easy remark that the inclusion (\ref{eq:inclusion1}) is equivalent to remove the row corresponding to $H_1^{P_0}$ from the second boundary matrix. 

\subsection{On the matrix $Mat(\partial_2)$ for sharp arrangements}\label{reduction}

Following section \ref{inclusion}, we denote by $M$ the submatrix of $Mat(\partial_2)$ obtained removing the row $H_1^{P_0}.$
We denote by $M_1$ the submatrix of $M$ with columns blocks $c^P,\,P\in \mathcal{P},$ and by $M_2$ the submatrix with columns $c^P,\,P\in L_2(\A) \backslash \mathcal{P}:$
\begin{equation}\label{Matdec}
M =( M_1 \mid M_2 ) \quad .
\end{equation}
We denote by $\partial_2(S(P))$ the sub-matrix of $M$ with columns $c^P$ and rows $H \in S(P).$ Similarly, we denote by $\partial_2(\bold{\widehat{U}}(P))$ the sub-matrix of $M$ with column block $c^P$ and rows $H\in \bold{\widehat{U}}(P).$

Performing operations  
$O^P_{i+1,i}(t^{-1}):=H_i^P- t^{-1}H_{i+1}^P,$ $1 \leq i \leq m(P)-1,$ we can diagonalize $\partial_2(S(P))$ as (see among others \cite{gaiffi2009morse}) 
\begin{center}
$D(P)=\left(\begin{array}{cccccc}
1-t^{m(P)}&0 &\hdots & 0\\
0&  & &0\\
\vdots &\ddots&&\vdots\\
& & 1-t^{m(P)}&0\\
 0& \hdots & 0&1-t\\
\end{array}\right)$
\end{center}
where the last row is obtained simplifying each entry by elementary columns operations with the last one. Since for any $P\in \mathcal{P},$ $\bold{\widehat{U}}(P)=\emptyset,$ the following lemma holds:

\begin{lemma}\label{lemdiag2}
$M_1$ is diagonalizable in $D(M_1)$ whose diagonal blocks are $D(P_1)$, $\hdots$, $D(P_m),$ and entries corresponding to rows $H_k^{P_0},\,1< k < m(P_0),$ are all zeros.
\end{lemma}

\begin{remark}\label{rkdiag} Since the monodromy has to divide $|\A|+1,$ if $P \in \mathcal{P}$ is a point such that $m(P)$ is coprime with $|\A|+1,$ then $1-t^{m(P)}$ cannot appear in the diagonal form of $M$, that is there exist columns operations on $M$ such that $D(P)=\diag(1-t),$ where 
$\diag(1-t)=\left(\begin{array}{cccc}
1-t &0 \hdots & 0\\

\vdots &\ddots&\vdots\\

 0& \hdots  &1-t\\
\end{array}\right).$\\
Denote by $\A_p =\{H \in (\A \setminus \{H_{1}^{P_0}\})  \mid \gcd(m(H \cap H_{1}^{P_0}), |\A|+1) \neq 1\}\cup \{H_k^{P_0}\}_{1 < k < m(P_0)}$ and 
$\A_{cp}=\A \setminus (\A_p \cup \{H_{1}^{P_0}\})$ two subsets corresponding to rows of $M$ and by $M_p$ and $M_{cp}$ the submatrices of $M$ corresponding, respectively, to rows $H\in \A_p$ an $H \in \A_{cp}$. Then the rows of $M$ can be re-ordered as follows: 
$$
M  ->  {{M_p}\choose{ M_{cp}} }.
$$
Let's now consider only the columns of $\left( \begin{array}{c} M_p \\ M_{cp} \end{array}\right)$ corresponding to the columns of the matrix $M_1$ in equation (\ref{Matdec}). After diagonalization process of $M_1$ in Lemma \ref{lemdiag2}, the matrix $D(M_1)$  can be re-ordered to keep the diagonal form as 
$$D(M_1) ->\left( \begin{array}{cc} D_p & 0 \\ 0 & 0 \\ 0 & D_{cp}  \end{array}\right),$$
where the first block corresponds to submtarix of rows $H\in \A_p \setminus \{H_k^{P_0}\}_{1 < k < m(P_0)}$, the zero block $( 0 \quad 0 )$ to rows in $\{H_k^{P_0}\}_{1 < k < m(P_0)}$ and the last block to rows $H \in \A_{cp}$. 
Then for any fixed row $H \in \A_p$ of the matrix $M_p$, the operations $O_{H,H^\prime}^{P,P^\prime}, row( H^\prime ) > row ( H )$ in the matrix $\left( \begin{array}{c} M_p \\ M_{cp} \end{array}\right)$ change, in the matrix $\left( \begin{array}{cc} D_p & 0 \\ 0 & 0 \\ 0 & D_{cp}  \end{array}\right)$, the column $c(H)$ of $\left(\begin{array}{c} D_p \\ 0  \\0 \end{array}\right)$ with non zero entry in the row $H$. Since this column is exchanged, in the diagonalization process, with the new reduced column $c^P_H$ with all entries $0$ except the one at row $H$, we can conclude that the diagonalization of rows $H \in \A_p$ of matrix $\left( \begin{array}{c} M_p \\ M_{cp} \end{array}\right)$ leaves the submatrix $( 0  \quad D_{cp} )$ unchanged. Hence, since all non zero entries of $D_{cp}$ are of the form $1-t^m$, $\gcd(m,|\A|+1) = 1$, it follows that, after diagonalization of the first rows $H \in \A_p$ of $\left( \begin{array}{c} M_p \\ M_{cp} \end{array}\right)$, independently of how rows $H \in \A_{cp}$ are changed by this half diagonalization, there exist columns and rows operations involving rows $H \in \A_{cp}$ to diagonalize $( 0  \quad D_{cp} )$ as $( 0  \quad \diag (1-t) )$. That is we can reduce to study only the submatrix $M_p$ of $M$ corresponding to rows $H\in \A_p$. Remark that if $H' \in \A_p$, then $ row( H^\prime ) > row ( H )$ if and only if $H' \rhd H$, that is in the matrix $M_p$ we perform operations $O_{H,H^\prime}^{P,P^\prime}, H^\prime \rhd H $.  
\end{remark}

Remark \ref{rkdiag} allows to immediately retrieve the following Libgober's result (see \cite{lib-mil}).

\begin{theorem}[]\label{thmLib}
Let $\A$ in $\R^2$ be a sharp arrangement with Milnor fiber $F$ and order $n=|\A|.$ If $\gcd(m(P),n+1)=1,$ for all $P \in \mathcal{P},$ then $\A$ is a-monodromic.
\end{theorem}
With more non trivial computations on $M,$ it is also possible to retrieve the following Yoshinaga's result proved in \cite{Yoshi2,Yoshi3} and also in \cite{Wil}.

\begin{theorem}[]\label{thmYosh}
Let $\A$ in $\R^2$ be a sharp arrangement with Milnor fiber $F$ and order $n=|\A|.$ Assume that $\mathcal{P}$ contains exactly one point $P$ such that $\gcd\{m(P),n+1\} \neq 1.$ Then $\A$ is a-monodromic.
\end{theorem}

In the end of this section we describe the effect of the diagonalization of $M_1$ on the right part $M_2$ of $M.$ Denote by $D(M_2)$ the new matrix obtained from $M_2$ after diagonalization of $M_1$.

\begin{theorem}\label{thmeffect}
Operations $O^P_{i+1,i}(t^{-1})=H_i^P- t^{-1}H_{i+1}^P,\,P\in \mathcal{P},$ change the  matrix $M_2$ as follows:
\begin{enumerate}
\item the zero entries in the row $H_i^{P}$ become entries of the form $t^{-1}S(P')$ in the columns corresponding to the block $c^{P'}$ if and only if $H_{i+1}^{P}=H_1^{P'}$;
\item submatrices $\partial_2(S(P')),\,P'\in L_2(\A)\backslash \mathcal{P},$ and rows $H_k^{P_0},\,1<k<m(P_0),$ are unchanged;
\item rows in $\partial_2(\bold{\widehat{U}}(P')),\,P'\in L_2(\A)\backslash \mathcal{P},$ become all zeros except the rows corresponding to $H \in \{\max_{\pol}S(P) \cap \bold{\widehat{U}}(P') \mid P\in \mathcal{P}\}$ and to $H_k^{P_0},1<k<m(P_0).$
\end{enumerate}
All other  entries are unchanged. 
\end{theorem}

\proof Using usual notation, $P$ will denote a point in $\mathcal{P}$ and $P'$ a point in $L_2(\A)\backslash \mathcal{P}.$
Recall that given a row $H$, the entries in the column block $c^{P'}$ are non zero if and only if $H \in S(P') \cup \bold{\widehat{U}}(P')$. Let us now prove 1., 2. and 3. separately.
\begin{enumerate}
\item Let's the entries of row $H_i^{P}$ and $H_{i+1}^{P}$ be respectively zero and not zero in the block $c^{P'}$, i.e. $H_i^P \notin S(P') \cup \bold{\widehat{U}}(P')$ and $H_{i+1}^P \in S(P') \cup \bold{\widehat{U}}(P')$. By simple geometric remark (see for instance lines $H_i^{P_m}, i>2$, in Figure \ref{fig1}) it follows immediately that if the line $H_i^P \notin S(P') \cup \bold{\widehat{U}}(P')$, then $H_{i+1}^P \notin \bold{\widehat{U}}(P')$. Hence $H_{i+1}^P \in S(P')$ and since if $H_{i+1}^P \in S(P')$ then $H_i^P \in U(P')$, the only possibility left for $H_{i+1}^P$ to be in $S(P')$ when $H_i^P \notin S(P') \cup \bold{\widehat{U}}(P')$ is that $H_i^P  \notin \Cone(P')$, that is if and only if $H_{i+1}^{P}=H_1^{P'}$.
\item Assume $H_i^P \in S(P').$ It is obvious that $H_{i+1}^P \notin S(P') \cup \bold{\widehat{U}}(P')$ and the entries $e(H_i^P,c^{P'})$ in the block $c^{P'}$ are unchanged by the operation $O^P_{i+1,i}(t^{-1}).$ Case of rows $H_k^{P_0},\,1<k<m(P_0),$ is trivial. 
\item  Assume $H_i^P \in \bold{\widehat{U}}(P')$. Then the entries $e(H_i^P,c^{P'})$ are changed by the operation $O^P_{i+1,i}(t^{-1})$ if and only if $e(H_{i+1}^P,c^{P'}) \neq 0$, that is if and only if one of the following two cases occurs:
\begin{enumerate}
\item $H_{i+1}^P \in \bold{\widehat{U}}(P')$ or
\item $H_{i+1}^P \in S(P').$
\end{enumerate}
In case (b) $H_i^P = \max_{\pol}S(P) \cap \bold{\widehat{U}}(P')$ and the entries $e(H_i^P,c^{P'})$ will be modified by the operation $O^P_{i+1,i}(t^{-1}).$ In particular, they (may) stay different from zero. Let us now study case (a). Then $H_i^P \neq \max_{\pol}S(P) \cap \bold{\widehat{U}}(P')$ and we can compute that $e(H_i^P-t^{-1}H_{i+1}^P,c^{P'})=0$ as follows. Let us denote by 
\begin{center}
$\alpha_i^P(P'):= |\{ H \pol H_i^P\,|\,H \in U(P')\}|,$\\
$l_i^P(C_j^{P'}):= |\{ H \pol H_i^P\,|\,H \in L(C_j^{P'}))\}|,$ and\\
 $u_i^P(C_j^{P'}):= |\{ H \pol H_i^P\,|\,H \in U(C_j^{P'}))\}|.$
\end{center}

With these notations we can express the $\bold{\widehat{U}}(P')$ addendum in formula (\ref{eq:boundary2}) of $\partial_2(l.e_{[C_j^{P'} \prec P']})$ as follows:

$$
\sum_{\substack{[C_{i-1}^{P} \prec F_i^{P}] \,s.t. \\H_i^{P} \in \bold{\widehat{U}}(P')}} t^{\alpha_i^{P}(P')} \cdot (1-t^{l_i^{P}(C_j^{P'})}) \cdot (t^{u_i^{P}(C_j^{P'})} - t^{m(P') - j })(l).e_{[C_{i-1}^{P} \prec F_i^{P}]} \quad.
$$

Given a critical $2-$cell $[C_j^{P'} \prec P'],$ as $H_{i}^P, H_{i+1}^P \in \bold{\widehat{U}}(P')$,  it is easily seen that $\alpha_{i+1}^P(P')= \alpha_i(P')+1,\,l_{i+1}^P(C_j^{P'})=l_i^P(C_j^{P'}),$ and $u_{i+1}^P(C_j^{P'})=u_i^P(C_j^{P'}).$ We deduce directly that the entries $e(H_i^P-t^{-1}H_{i+1}^P,c^{P'})$ are all zeros.
\end{enumerate}
\endproof
In view of the new non zero coefficients of type 1. in Theorem \ref{thmeffect} that appear when we diagonalize $M_1$, define for any $P'\in L_2(\A)\backslash \mathcal{P}$:
$$
\mathcal{N}(P'):=\{H_i^P,\,P\in \mathcal{P}\,|\,H_{i+1}^P = H_1^{P'}\}.
$$
Notice that $\mathcal{N}(P')$ is either the empty-set or contains just one line.

\subsection{On rows $H_k^{P_0}$ in the matrix $M$}\label{simplification}

It is an easy remark that lines of the form $H_k^{P_0}$ are never in $\Cone(P)$, unless $P$ belongs to one of them (see for instance points $P$ and $P'$ in Figure \ref{fig1}). On the other hand, if they belong to the cone of $P$, by our choice of polar coordinates, they also belong to the upper of $P$. More precisely, if $P \in L_2(\A)$, the following holds:
$$
H_k^{P_0} \in \bold{\widehat{U}}(P) \mbox{ if and only if } \exists~h > k \mbox{ s.t. }  H_h^{P_0} \in S(P) \quad.
$$ 
The first part of the following Lemma is then proved.

\begin{lemma}\label{up0}
Let  $P \in L_2(\A)$ be a point, then:
\begin{enumerate}
\item if $H_{k}^{P_0} \notin S(P)$, for all $k,\,1\leq k\leq m(P_0)-1$, then $e(H_{k}^{P_0},c^{P})=0,$ for all $k,\,1 \leq k \leq m(P_0)-1$;
\item if $H_{m(P_0)-1}^{P_0} \in S(P),$ then for any other line $H \in S(P)$, if $c^P_H$ is the column defined in Notation \ref{not:goodcolumn}, we have:
\begin{equation*}
\begin{split}
& e(H_{m(P_0)-1}^{P_0},c_H^P)=t^{m(P)-1}-1, \mbox{ and }\\
& e(H_{k}^{P_0},c_H^P)=t^{m(P_0)-k-2}(1-t)(1-t^{m(P)-1}),\, 1< k < m(P_0)-1.
\end{split}
\end{equation*}
\end{enumerate}

\end{lemma}
\proof 
Point 1. is trivial. Let's prove point 2. \\
Assume $H_{m(P_0)-1}^{P_0}\in S(P)$. Then $H_{k}^{P_0}\in \bold{\widehat{U}}(P)$, for all $k < m(P_0)-1$,  and:
 \begin{enumerate}
\item[i)] $H_{m(P_0)-1}^{P_0}= H_1^P$, that is any other line $H \in S(P)$ is of the form $H=H_i^P$, $i >1$ and,  by Notation \ref{not:goodcolumn}, $c_H^P=c_1^p$;
\item[ii)] the entry $e(H_{m(P_0)-1}^{P_0},c_1^P)$ will be in the $S(P)$ addendum of formula  (\ref{eq:boundary2});
\item[ii)] the entry $e(H_{k}^{P_0},c_1^P)$ will be in the $\bold{\widehat{U}}(P)$ addendum of formula  (\ref{eq:boundary2})
\end{enumerate}
and we get:
\begin{equation}\label{SP}
e(H_{m(P_0)-1}^{P_0},c_1^P)= \bigg( \displaystyle{\prod_{\substack{H' < H_{m(P_0)-1}^{P_0}~s.t.\\ H'\in U(P)}}} t \bigg) \bigg[\prod_{\substack{H'~s.t.\\ H' \in[C_1^P\rightarrow H_{m(P_0)-1}^{P_0})}}t -\prod_{\substack{H' < H_{m(P_0)-1}^{P_0}~s.t.\\ H' \in S(P)}}t\bigg] 
\end{equation}
\begin{equation}\label{UCP}
e(H_k^{P_0},c_1^{P}) = \bigg(\displaystyle{\prod_{\substack{H' < H_{k}^{P_0}~s.t.\\H'\in U(P)}}}t \bigg)\bigg(1-\prod_{\substack{H' < H_{k}^{P_0}~s.t.\\ H' \in L(C_1^P)}}t \bigg)\bigg(\prod_{\substack{H' < H_{k}^{P_0}~s.t.\\H'\in U(C_1^P)}}t -\prod_{\substack{H'~s.t.\\H'\in U(C_1^P)}}t\bigg)
\end{equation} 
Since $H_{m(P_0)-1}^{P_0}$ is the smallest line of $\A$ and $[C_1^P\rightarrow H_{m(P_0)-1}^{P_0}) = \{H_2^P,\hdots,H_{m(P)}^P\},$ equation (\ref{SP}) becomes $e(H_{m(P_0)-1}^{P_0},c_1^P)= t^{m(P)-1}-1.$\\
Since  $\{H' < H_{k}^{P_0}~s.t.~H'\in U(P)\} = \{H_{m(P_0)-1}^{P_0} < H' < H_{k}^{P_0}\}$, its cardinality is $m(P_0)-k-2.$ Moreover $|\{H' < H_{k}^{P_0}~s.t.~H' \in L(C_1^P)\}|=1$ as $L(C_1^P)=\{H_{m(P_0)-1}^{P_0}\}$. Finally, any line in $U(C_1^P)=\{H_2^P,\hdots,H_{m(P)}^P\}$ is bigger than $H_k^{P_0},$ and equation (\ref{UCP}) becomes $e(H_{k}^{P_0},c_1^P)=t^{m(P_0)-k-2}(1-t)(1-t^{m(P)-1})$.
\endproof

\begin{lemma}\label{up0h1}
Let $H=H_{m(P_0)-1}^{P_0} $ and $P\in H.$ We have that:
\begin{equation*}
\begin{split}
& e(H,c_H^P)=t-1, \mbox{ and }\\
& e(H_{k}^{P_0},c_H^P)=t^{m(P_0)-k-2}(1-t)^2, 1< k < m(P_0)-1.
\end{split}
\end{equation*}
\end{lemma}
\proof
Analogously to Proof of Lemma \ref{up0},  $H=H_{m(P_0)-1}^{P_0}=H_1^P$. Then $c_H^P= c_{m(P)-1}^P$ and since $H_{m(P_0)-1}^{P_0}$ is the smallest line of $\A$ and $[C_{m(P)-1}^P\rightarrow H_{m(P_0)-1}^{P_0}) = \{H_{m(P)}^P\},$ equation (\ref{SP}) with $c_{m(P)-1}^P$ instead of $c_{1}^P$ becomes $e(H_{m(P_0)-1}^{P_0},c_{m(P)-1}^P)= t-1.$\\
About second equality, all lines in  $L(C_{m(P)-1}^P)=\{H_1^P,\hdots,H_{m(P)-1}^P\}$ but $H_1^P$ are bigger than $H^{P_0}_k$, $k < m(P_0)-1$, that is  $|\{H' < H_{k}^{P_0}~s.t.~H' \in L(C_{m(P)-1}^P)\}|=1$. Finally, as $U(C_{m(P)-1}^P)=\{H_{m(P)}^P\},\, H_{m(P)}^P > H_{k}^{P_0}$, equation (\ref{UCP}) with $c_{m(P)-1}^P$ instead of $c_{1}^P$ becomes $e(H_{k}^{P_0},c_{m(P)-1}^P)=t^{m(P_0)-k-2}(1-t)^2.$

\endproof

\begin{remark}\label{eliminationP0}The columns $c_H^P$ are key columns in the diagonalization process. Then we are interested in  studying and simplifying the non zero entries on those columns. 
Clearly by Lemmas \ref{up0} and \ref{up0h1} performing operations 
\begin{equation}\label{op:P0}
O^{P_0}_{m(P_0)-1,k}(t^{m(P_0)-k-2}(t-1))=H_k^{P_0} + t^{m(P_0)-k-2}(1-t)H_{m(P_0)-1}^{P_0}
\end{equation} 
we have that for any point P and line $H \in S(P), H \neq H_k^{P_0},1<k<m(P_0)-1$, the entries $e(H_k^{P_0},c_H^P),1<k<m(P_0)-1,$ become all zero.
All the other columns of the matrix are unchanged by these operations, since they have zeros entries on the row $H_{m(P_0)-1}^{P_0}.$
\end{remark}

Theorem \ref{thmeffect} and the previous discussion motivate the following definition.

\begin{definition}\label{def:Uppmax}
Let $P\in L_2(\A) \backslash \mathcal{P}$ be an affine point. We define the upper cone maximal of $P$ as the set
$$
\bold{\widehat{U}}_{Max}(P):=\{\max_{\pol} \bold{\widehat{U}}(P) \cap S(P_i) \mid P_i \in \mathcal{P}\} \cup \mathcal{H}_0(P),
$$
where $\max_{\pol} \bold{\widehat{U}}(P) \cap S(P_i)$ is the maximal hyperplane with respect to the order $\pol$ in each $S(P_i)$ which also belongs to the upper cone $\bold{\widehat{U}}(P)$, and $\mathcal{H}_0(P)$ is the following subset of $S(P_0)\cap \bold{\widehat{U}}(P):$
$$\mathcal{H}_0(P):= \left \{
\begin{split} &\{H_k^{P_0}\}_{1 <k < m(P_0)-1} \cap  \bold{\widehat{U}}(P), & \mbox{ if \,} H_{m(P_0)-1}^{P_0} \notin S(P)\\
&\emptyset  & \mbox{ if \,} H_{m(P_0)-1}^{P_0} \in S(P). 
\end{split}
\right .
$$
\end{definition}

\begin{remark}\label{simplification1} Since $H_k^{P_0} \in \bold{\widehat{U}}(P) \mbox{ if and only if } \exists~h > k \mbox{ s.t. }  H_h^{P_0} \in S(P)$, the condition $\mathcal{H}_0(P) \neq \emptyset$ is equivalent to the existence of an $h \neq m(P_0)-1$ such that $H_h^{P_0} \in S(P)$. 
Hence after we performed operation in equation (\ref{op:P0}), the only non zero entries of rows $H_k^{P_0}$ are in the column blocks $c^P$ such that $H_k^{P_0} \in S(P)\cup \bold{\widehat{U}}_{Max}(P).$ That is, any row operation $H-\varphi H_k^{P_0}$ leaves the entry $e(H,c_H^P)$ unchanged for any point $P$ such that  
$H_k^{P_0} \notin S(P)\cup \bold{\widehat{U}}_{Max}(P).$
\end{remark}
We will use this Remark \ref{simplification1} in section \ref{subgraph1} (respectively \ref{subgraph2}) in order to simplify the rows $H_k^{P_0},\,1<k<m(P_0)-1$ (respectively $H_k^{P_0},\,2<k<m(P_0)-1$).

\section{Graph and monodromy of the Milnor fiber}\label{sec:mains}

In this section, we study special graphs obtained from $\mathcal{G}(\A)$ in order to study the monodromy of $\A$. In Section \ref{sec:sharp}  we essentially proved the following facts:
\begin{enumerate}
\item[i)] the line $H_1^{P_0}$ can be removed (see section \ref{inclusion}); 
\item[ii)] performing operations $O^P_{i+1,i}(t^{-1})$, we diagonalize the matrix $M_1$ getting entries $e(H_i^{P},c_{i-1}^P)=1-t^{m(P)}$, $1 < i < m(P)$, $e(H_{m(P)}^{P},c_{m(P)-1}^P)=1-t$, for all $P \in \mathcal{P}$ (see section \ref{reduction});
\item[iii)]after performing operations $O^P_{i+1,i}(t^{-1})$ and $O^{P_0}_{m(P_0)-1,k}(t^{m(P_0)-k-2}(t-1)),$ we have that
\begin{center}
$e(H,c^{P'})\neq 0 \Leftrightarrow H \in S(P')\cup \bold{\widehat{U}}_{Max}(P') \cup \mathcal{N}(P')$
\end{center}
(see Theorem \ref{thmeffect} and Remark \ref{eliminationP0}).
\end{enumerate}
Since the only possible monodromies have to divide both $|\A|+1$ (by the equivariant decomposition in equation (\ref{eq:decomposition}) ) and $m(P_0)$ ( see \cite[3.21 (ii)]{Yoshi2}), it follows from  i) and ii) that the hyperplanes
\begin{equation}\label{eq:remove}
\A_r=\{H_i^P \in \A | P \in \mathcal{P} \mbox{ and } \gcd(m(P),|\A|+1)=1 \mbox{ or } \gcd(m(P),m(P_0))=1 \mbox{ or } i=m(P) \} 
\end{equation} 
can all be removed in the sense of Definition \ref{def:remove} by the same argument than in Remark \ref{rkdiag}. Notice that $ i=m(P)$ includes the case of double points, i.e. $m(P)=2$. \\
Then in order to study the monodromy of $\A$ we can reduce to study the graph $\mathcal{G}_{\A'}$ defined as follows: 
\begin{enumerate}
\item vertices of $\mathcal{G}_{\A'}$ are vertices $(H,P)$ of $\mathcal{G}(\A)$ such that $H \in \A'=\A \setminus \A_r$ and $P \in L_2(\A) \setminus \mathcal{P}$;
\item two vertices $(H,P)$ and $(H',P^\prime)$ are connected by an edge from $(H,P)$ to $(H',P^\prime)$ if and only if $H^\prime \in S(P)\cup \bold{\widehat{U}}_{Max}(P) \cup \mathcal{N}(P)$.
\end{enumerate} 
Let us remark that $\mathcal{G}_{\A'}$ has lesser vertices than $\mathcal{G}(\A)$ and all edges $[(H,P),(H',P')]$ for which $H' \in \bold{\widehat{U}}(P) \setminus \bold{\widehat{U}}_{Max}(P)$ disappear. Remark that $\mathcal{G}_{\A'}$ is not properly a subgraph of $\mathcal{G}(\A)$ since new edges $[(H,P),(H',P')]$ corresponding to $H'\in \mathcal{N}(P)$ appear. Graph $\mathcal{G}_{\A'}$ is more informative than $\mathcal{G}(\A)$ but still too big. It can be highly reduced as we will show in the two next subsections. 

\begin{notation}\label{not:arrnozero} Denote by
$$
\A_0=\A' \setminus \{H^{P_0}_k\}_{1< k <m(P_0)-1}
$$
and by
$$
\A'_0=\A' \setminus \{H^{P_0}_k\}_{2< k <m(P_0)-1} \quad .
$$ 
\end{notation}

Note that $\A_0$ is obtained from $\A'_0$ simply removing the hyperplane $H_2^{P_0}$. 

\subsection{The subgraph $\mathcal{G}_{\A_0,\last}$ of $\mathcal{G}_{\A'}$}\label{subgraph1}

Recall that given a line $H$ we call $\last(H)\in L_2(\A)\backslash \mathcal{P}$ (resp. $\mini(H) \in L_2(\A)\backslash \mathcal{P}$) the last point (resp. the second point) in $H$ with respect to the order induced by the polar order $\pol$ of $\A.$

\begin{remark}\label{lastmininfty}
By basic geometric remarks we have:
\begin{enumerate}
\item if $2 < k < m(P_0)-1,$ then $\last(H_k^{P_0}) \neq \last(H)\,\text{and}\,\mini(H),$ for any $H \neq H_k^{P_0};$
\item if $\last(H_2^{P_0})= \last(H),$ for a certain $H \neq H_2^{P_0},$ then $m(P_0)=3$ and $H=H_2^{P_i},$ for a certain $P_i\in \mathcal{P}.$ 
\item if $\mini({H_2}^{P_0}) = \last(H),$ for a certain $H \neq H_2^{P_0},$ then $m(P_0)=3;$ 
\item if $\mini(H_2^{P_0})=\mini(H_2^{P_i}),$ for a certain $P_i\in \mathcal{P},$ then $m(P_i)=2;$
\item if $\last(H_{m(P_0)-1}^{P_0})=\last(H),$ for a certain $H\neq H_{m(P_0)-1}^{P_0},$ then $H=H_2^{P_i},$ for a certain $P_i\in \mathcal{P}.$
\end{enumerate} 
\end{remark}

The following remark studies the contribution of the lines $H_k^{P_0}$ in the non zero entries of the matrix $D(M_2)$ related to columns $c^{\last(H_j^{P_i})}$ or $c^{\mini(H_2^{P_i})},\,P_i\in \mathcal{P}.$
\begin{remark}\label{rksimplast}The following facts hold:
\begin{enumerate}
\item for any $H \in \A$, if $H_k^{P_0} \in S(\last(H))$, then $k=m(P_0)-1$ and 
\begin{equation*}
H_{m(P_0)-1}^{P_0} \in S(\last(H)) \Leftrightarrow H_k^{P_0}\in \bold{\widehat{U}}(\last(H)),1 \leq k < m(P_0)-1,
\end{equation*}
that is $H_k^{P_0} \notin S(\last(H_j^{P_i})) \cup \bold{\widehat{U}}_{Max}(\last(H_j^{P_i})) \cup \mathcal{N}(\last(H_j^{P_i})),$ for any $j$ and $P_i\in \mathcal{P}$;
\item let $H_2^{P_i}\in \A,\,P_i\in \mathcal{P}.$ Since $\mini(H_2^{P_i})$ is smaller or equal than $H_2^{P_i} \cap H_2^{P_0},$ for any $2 < k < m(P_0)$ we have that 
\begin{equation*}
H_k^{P_0}\notin S(\mini(H)) \cup \bold{\widehat{U}}_{Max}(\mini(H))\cup \mathcal{N}(\mini(H)) .
\end{equation*}
\end{enumerate}

\end{remark}

\begin{lemma}\label{removeHP0last}Rows $\{H^{P_0}_k\}_{1< k <m(P_0)-1}$ can be removed in the sense of Definition \ref{def:remove} without changing columns $c_H^{\last(H)}$, $H \in \A'$.
\end{lemma}

\begin{proof}Let's remove rows $H_k^{P_0}$, $1< k <m(P_0)-1,$ in the matrix $D(M)$ by using the last points $\last(H_k^{P_0})$ and the corresponding columns $c_{H_k^{P_0}}^{\last(H_k^{P_0})}$ defined in Notation \ref{not:goodcolumn}. It is clear that $e(H',c_{H_k^{P_0}}^{\last(H_k^{P_0})})=0$ for all $H'\pol H_k^{P_0}.$ In order to remove the entries $e(H',c_{H_k^{P_0}}^{\last(H_k^{P_0})})\neq 0$ with $H'\rhd H_k^{P_0},$ we perform rows operations by using the entry $e(H_k^{P_0},c_{H_k^{P_0}}^{\last(H_k^{P_0})})=t^\alpha(1-t).$ It follows from Remarks \ref{simplification1} and  \ref{rksimplast} that these row operations do not affects the other columns $c_H^{\last(H)}$ of $D(M).$
We have that the last points $\last(H_k^{P_0}),\,1<k<m(P_0)-1,$ are different from all the other last points $\last(H)$ for the remaining line $H \in \A'$, see Remark \ref{lastmininfty} 1. and 2.
\end{proof}

\begin{definition}We call $\mathcal{G}_{\A_0,\last}$ the subgraph of $\mathcal{G}_{\A'}$ with vertices $$(H,\last(H)),\,\,\text{for all} \,\,H\in \A_0.$$
\end{definition}

\begin{notation} For the sake of simplicity, we will only use the hyperplane $H$ instead of $(H,\last(H))$ for the vertices of $\mathcal{G}_{\A_0,\last}$ and the notation $[H,H']$ for the edge oriented from $(H,\last(H))$ to $(H',\last(H')).$
\end{notation}

Our goal is to show that, under special conditions, in order to study the monodromy of $\A$ it is enough to study the graph $\mathcal{G}_{\A_0,\last}$. 
Let's start by observing that since $\A$ is sharp, all intersections lie in the same side of $H_1^{P_0},$ and then for any line $H_h^{P_i}\rhd H_k^{P_j}$, $P_i,P_j \in \mathcal{P}$, we have that $H_h^{P_i}\cap H_k^{P_j}\rhd H_{h+1}^{P_i}\cap H_k^{P_j}.$  Viceversa, for any line $H_h^{P_i}\pol H_k^{P_j},$ we have that $H_h^{P_i}\cap H_k^{P_j}\pol H_{h+1}^{P_i}\cap H_k^{P_j}$ (see Figure \ref{fig1}). By this we get the following important remark on edges of $\mathcal{G}_{\A_0,\last}$.

\begin{remark}\label{rk1}
By $\A$ being a sharp arrangement, the following facts hold for $\last(H)$, $H \in \A_0$, which allows us to list all the possible edges appearing in $\mathcal{G}_{\A_0,\last}$. As usual let $P_i$, $ 1 \leq i \leq |\mathcal{P} |$, denote points in $\mathcal{P}$ with $H^{P_i} \pol H^{P_j}$ if and only if $i<j$: \\
\begin{enumerate}
\item $H_h^{P_i} \parallel H_k^{P_j}\,\,\Rightarrow \,\, \left \{ \begin{array}{cc} 
h=2,\,k=m(P_j) & \mbox{ if } i>j \\
h=m(P_i),\,k=2 & \mbox{ if } i<j ; 
\end{array}
\right .$
\item if $H_h^{P_i}\in S(\last(H_k^{P_j}))$ then:
\begin{equation}\label{eq:lastp}
\left \{ \begin{array}{cccc} 
h=2 & \mbox{ or } & h=3 \mbox{ and } H_{2}^{P_i} \parallel H_k^{P_j} & \mbox{ if } i>j \\
h=m(P_i) & \mbox{ or } & h=m(P_i)-1 \mbox{ and } H_{m(P_i)}^{P_i} \parallel H_2^{P_j} & \mbox{ if } i<j ;
\end{array}
\right .
\end{equation}
since condition  $H_{2}^{P_i} \parallel H_k^{P_j}$ implies $k=m(P_j)$ and $H_{m(P_i)}^{P_i} \notin \A_0$, conditions (\ref{eq:lastp}) correspond respectively to edges  
$E_1=[H_k^{P_j},H_{2}^{P_i},\pol]$ and $E_2=[H_2^{P_j},H_{m(P_i)-1}^{P_i},\rhd]$ in  $\mathcal{G}_{\A_0,\last}$. Notice that condition $H_{m(P_i)}^{P_i} \parallel H_2^{P_j}$ is empty for $i=0$ as $H_{m(P_0)}^{P_0}$ is the line at infinity. Indeed we can have that
$H_{m(P_0)-1}^{P_0}\in S(\last(H_k^{P_j}))$  for any $k,j$, $j \neq 0$, that is
$E_3=[H_k^{P_j},H_{m(P_0)-1}^{P_0},\rhd]$ is another possible edge in $\mathcal{G}_{\A_0,\last}$.
\item If $H_h^{P_i}\in \bold{\widehat{U}}_{Max}(\last(H_k^{P_j})),$ then $0<i \leq j.$ Indeed, if $j<i$ then $H_h^{P_i} \in U(\last(H_k^{P_j})) \Leftrightarrow H_h^{P_i} \parallel H_k^{P_j}$, that is $k=m(P_j)$ and $H_{k}^{P_j}\notin \A_0.$ The following hold:
\begin{enumerate}
\item if $i=j,$ then $h=k-1$ and the corresponding edge in $\mathcal{G}_{\last,\A_0}$ is $E_4=[H_k^{P_j},H_{k-1}^{P_j},\rhd]$ (remark that this case also includes the case $H_{k-1}^{P_j} \in \mathcal{N}(\last(H_k^{P_j}))$;
\item if $0<i<j$, then we are in one of the following situations (by $H_{m(P_i)}^{P_i}\notin \A_0$):
\begin{enumerate}
\item $h=m(P_i)-1$ with $H_{m(P_i)}^{P_i}\in S(\last(H_k^{P_j}))$ or $H_{m(P_i)}^{P_i} \parallel H_k^{P_j}$ (i.e. $k=2$), and the corresponding edge is $E_5=[H_k^{P_j},H_{m(P_i)-1}^{P_i},\rhd]$; 
\item $h=m(P_i)-2$ with $H_{m(P_i)}^{P_i} \parallel H_k^{P_j}$ (i.e. $k=2$), $H_{m(P_i)-1}^{P_i}\in S(\last(H_2^{P_j}))$, and the corresponding edge is $E_6=[H_2^{P_j},H_{m(P_i)-2}^{P_i},\rhd];$
\end{enumerate}
\end{enumerate}
\item if $H_{m(P_i)-1}^{P_i}\in S(\last(H_2^{P_j})),\, i <j$, then $\last(H_{m(P_i)-1}^{P_i})=\last(H_2^{P_j}).$ Indeed in this case (see figure \ref{fig:last})
$$H \cap H_{m(P_i)-1}^{P_i}\pol \last(H_{m(P_i)-1}^{P_i}) \pol H \cap H_2^{P_j}  \Leftrightarrow H_{m(P_i)} \pol H \pol H_2^{P_j}$$
and, since $\A$ is sharp, by $H_{m(P_i)} \pol H \pol H_2^{P_j}$ it follows $H_{m(P_i)} \parallel H \parallel H_2^{P_j}$ and we finished. 

\begin{figure}[htbp]
\centering
\begin{tikzpicture}[scale=1]

\draw (0,0) node [left] {$H_1^{P_0}$} --(0,4);
\draw (1,1) node [right] {$H_{m(P_i)}^{P_i}$} --(-4,1);
\draw (1,3) node [right] {$H_2^{P_j}$} --(-4,3);
\draw (1,0) node [right] {$H_{m(P_i)-1}^{P_i}$} --(-3,4);

\draw (-2,3) node {$\bullet$}; 
\draw (-2,3) node[below left]{$\last(H_2^{P_j})$} ;

\end{tikzpicture}
\caption{$H_{m(P_i)-1}^{P_i}\in S(\last(H_2^{P_j}))$}
\label{fig:last}
\end{figure}
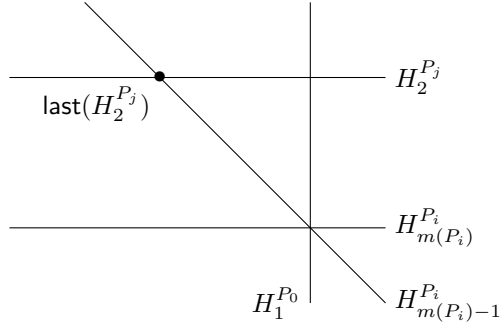

\end{enumerate}  
This last condition will play an important role in our main Theorem.
\end{remark}
By Remark \ref{rk1} we get that all possible edges in $\mathcal{G}_{\A_0,\last}$ are the one in Figure \ref{fig:edge_glast}.

\begin{figure}[htbp]
\centering
\begin{tikzpicture}

\tikzset{vertex/.style = {shape=rectangle}}
\tikzset{edge/.style = {->,> = latex'}}
\node[vertex] (h) at  (12.5,0) {$H_{k-1}^{P_j}$};
\node[vertex] (h-h1) at  (11,0.6) {$\rhd$};
\node[vertex] (h1) at  (9.5,0) {$H_k^{P_{j}}$};
\node[vertex] (h11-h2) at  (6,0.6) {$\rhd$};
\node[vertex] (h2) at  (4.5,0) {$H_{2}^{P_{j}}$};
\node[vertex] (h11) at  (7.8,0) {$H_{m(P_i)-1}^{P_{i}}$};

\node[vertex] (hk) at  (3,0) {$H_{2}^{P_{i}}$};
\node[vertex] (hk-hbar) at  (1.5,0.6) {$\pol$};
\node[vertex] (hbar) at (0,0) {$H_k^{P_j}$ };

\node[vertex] (hk1) at  (3,-2) {$H_{k}^{P_j}$};
\node[vertex] (hk1-hk2) at  (4.3,-1.5) {$\rhd$};
\node[vertex] (hk2) at  (6,-2) {$H_{m(P_i)-1}^{P_{i}}$};


\node[vertex] (hks) at  (9,-2) {$H_{2}^{P_{j}}$};
\node[vertex] (hks-htilde) at  (10.3,-1.5) {$\rhd$};
\node[vertex] (htilde) at  (12,-2) {$H_{m(P_i)-2}^{P_{i}}$};



\draw[edge] (hbar.east)  to[bend left] (hk.west);
\draw[edge] (h2.east)  to[bend left] (h11.west);
\draw[edge] (h1.east)  to[bend left] (h.west);
\draw[edge] (hk1.east)  to[bend left] (hk2.west);
\draw[edge] (hks.east)  to[bend left] (htilde.west);


\node[vertex] (hk-hbar) at  (1.5,-0.5) {\scriptsize{$E_1: 0 \leq j < i$}};
\node[vertex] (h11-h2) at  (6,-0.5) {\scriptsize{$E_2: 0 \leq i < j$}};
\node[vertex] (h-h1) at  (11,-0.5) {\scriptsize{$E_4: 0 < j$}};
\node[vertex] (hk1-hk2) at  (4.3,-2.7) {\scriptsize{$E_3 \mbox{ and } E_5: 0 \leq i < j$}};
\node[vertex] (hks-htilde) at  (10.3,-2.7) {\scriptsize{$E_6: 0 < i < j$}};

\end{tikzpicture}
\caption{Edges in $\mathcal{G}_{\A_{0},\last}$ }
\label{fig:edge_glast}
\end{figure}
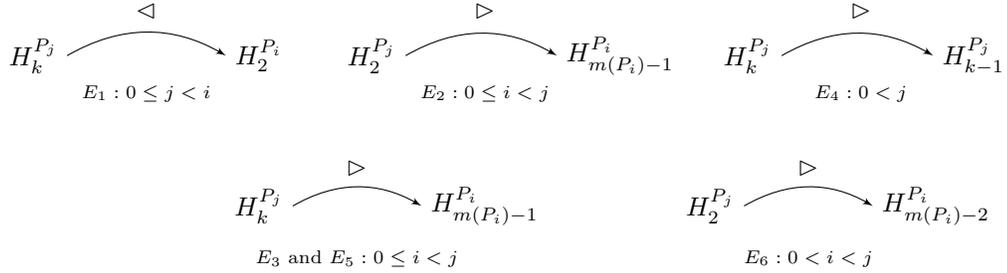


\begin{proposition}\label{prop1} Let $H \neq H'$ be two hyperplanes in $\A_0$  such that 
\begin{center}
$\last(H)= \last(H'),$
\end{center}
then $H,H' \in \{H_2^{P_i}, H_{m(P_i)-1}^{P_i}\}_{1 \leq i \leq |\mathcal{P}|} \cup \{H_{m(P_0)-1}^{P_0}\}.$
\end{proposition}
\proof Follows directly from Remark \ref{rk1}.
\endproof

In particular the two following  propositions hold.

\begin{proposition}\label{prop2}
Given a line $H_2^{P_i}\in \A_0,$ if there exists $H\in \A_0$ such that $\last(H)=\last(H_2^{P_i}),$ then we are in one of the following cases:
\begin{enumerate}
\item $H=H_{m(P_j)-1}^{P_j}$ with $0 \leq j < i.$ Moreover, if $j\neq 0,$ then $H_{m(P_j)}^{P_j} \parallel H_2^{P_i}.$
\item $H=H_{m(P_j)-1}^{P_j}$ with $0<i<j$ and $m(P_i)=m(P_j)=3$.
\item $H=H_2^{P_j}$ with $0<j<i$ and $m(P_j)=3$.
\item $H=H_2^{P_j}$ with $0<i<j$ and $m(P_i)=3.$ 

\end{enumerate}
\end{proposition}

\proof If $H=H_k^{P_j}$, then by Proposition \ref{prop1}, $k=2$ or $k=m(P_j)-1.$ If $k=2,$ we are in cases 3. and 4. and by remark \ref{rk1} 2., the smallest between $H_2^{P_j}$ and $H_2^{P_i}$ has to be the last minus one line in its point $P_i$ or $P_j,$ which implies $m(P_i)=3$ or $m(P_j)=3.$ If $k=m(P_j)-1,$ then by Remark \ref{rk1} 2. the biggest between $H_{m(P_j)-1}^{P_j}$ and $H_2^{P_i}$ has to be the second line in its point, while the smallest one has to be the last minus one. This directly imply points 1. and 2.  
\endproof

\begin{proposition}\label{prop3}
Given a line $H_{m(P_i)-1}^{P_i}\in \A_0,$ if $H\in \A_0$ is such that $\last(H)=\last(H_{m(P_i)-1}^{P_i}),$ then we are in one of the following situations:
\begin{enumerate}

\item $H=H_{m(P_j)-1}^{P_j}$ with $0 \leq j <i$ and $m(P_i)=3$.
\item $H=H_{m(P_j)-1}^{P_j}$ with $0 \leq i < j $ and $m(P_j)=3$.
\item $H=H_2^{P_j}$ with $0 \leq i < j.$ Moreover, if $i \neq 0$ then $H_{m(P_i)}^{P_i} \parallel H.$
\item $H=H_2^{P_j}$ with $0<j<i$ and $m(P_i)=m(P_j)=3$.
\end{enumerate}
\end{proposition}

\proof Similar to the proof of Proposition \ref{prop2}.
\endproof

In many situations, it happens that the graph $\mathcal{G}_{\A_{0},\last}$ has cycles. We will see that these cycles involve points $P_i \in \mathcal{P}$ of multiplicity 3. Hence, let's now define
\begin{equation}\label{eq:A03}
\A_{(0,3)}= \A_0 \backslash \{H_h^{P_i} |\,P_i\in \mathcal{P} \mbox{ and } m(P_i)=3\} .
\end{equation}
By above Propositions \ref{prop1}, \ref{prop2} and \ref{prop3} it follows that if, for hyperplanes in $\A_{(0,3)}$, 
\begin{equation}\label{cond:last}
\last(H_{m(P_i)-1}^{P_i}) \neq \last(H_2^{P_j}),  \mbox{ for all } 0 \leq i <j, 
\end{equation}
then  
$$| \mathcal{V}_0(\mathcal{G}_{\A_{(0,3)},\last})|=|\A_{(0,3)}|, $$
$| \mathcal{V}_0(\mathcal{G}_{\A_{(0,3)},\last})|$ being the number of vertices of $\mathcal{G}_{\A_{(0,3)},\last}$, and the following results holds.

\begin{proposition}\label{thm1}
Let $\A$ in $\R^2$ be a sharp arrangement such that $$\last(H_{m(P_i)-1}^{P_i}) \neq \last(H_2^{P_j}), \quad 0 \leq i <j, $$ 
holds for hyperplanes in $\A_{(0,3)}$. Then the matrix associated to $\mathcal{G}_{\A_{(0,3)},\last}$ can be diagonalized as $\diag(1-t).$
\end{proposition}

The following main theorem, stated in Introduction, follows from the previous proposition and remarks \ref{lastmininfty} and \ref{rkamono}.

\begin{theorem}\label{thm1.1} Let $\A \subset \R^2$ be a sharp arrangement such that $$\last(H_{m(P_i)-1}^{P_i}) \neq \last(H_2^{P_j}), \quad 0 \leq i <j, $$
holds for hyperplanes in $\A_{(0,3)}$. Then $\A$ is a- or 3-monodromic.
\end{theorem}

\begin{remark}\label{rem:diffpol} Condition $$\last(H_{m(P_i)-1}^{P_i}) \neq \last(H_2^{P_j}),  \quad 0 \leq i <j $$ 
for hyperplanes in $\A_{(0,3)}$ and, more in general, for hyperplanes in $\A_{0}$ strictly depends on the chosen polar coordinates system $(V_0,V_1)$. Indeed the only requirement is that lines $(H_{\infty},\overline{H}_1^{P_0})$ are sharp pair. This is equivalent to say that sufficient condition for (\ref{cond:last}) to hold is that it exists a polar coordinates system $(V_0,V_1)$ with lines $(H_{\infty},\overline{H}_1^{P_0})$ sharp pair such that (\ref{cond:last}) holds for this system.\\
In particular for any sharp pair there are four different natural choices that are the two different ways we can choose the line at infinity, that is if $(\overline{H},\overline{H}')$ is the sharp pair we can have
\begin{enumerate}
\item $(\overline{H},\overline{H}')=(H_{\infty},\overline{H}_1^{P_0})$;
\item $(\overline{H}',\overline{H})=(H_{\infty},\overline{H}_1^{P_0})$.
\end{enumerate} 
The other two options depend on the choice of the origin $V_0$ that can be 
\begin{enumerate}
\item[i.] in the chamber in the bottom left corner as in Figure \ref{fig1} 
\item[ii.] or in the chamber in the upper left corner as in Figure \ref{fig:example4}. 
\end{enumerate}
It is not difficult to check that for a given choice $(V_0,V_1)$ of the polar coordinates system such that $(H_{\infty},\overline{H}_1^{P_0})$ is sharp pair, the following four conditions on hyperplanes of $\A$ are equivalent to (\ref{cond:last}) with respect to the four different natural choices described above.
\begin{enumerate}
\item[1.i.] $\last(H_{m(P_i)-1}^{P_i}) \neq \last(H_2^{P_j}),\, 0 \leq i <j $;
\item[1.ii.]$\last(H_{3}^{P_i}) \neq \last(H_{m(P_j)}^{P_j})$ and $\mini(H_{m(P_0)-1}^{P_0}) \neq \last(H_{m(P_j)}^{P_j}), \, 0 \leq j <i $;
\item[2.i] $\mini(H_{2}^{P_j}) \neq \mini(H_{m(P_{j+1})}^{P_{j+1}})$ and $\mini(H_{2}^{P_0}) \neq \mini(H_{m(P_{j})}^{P_{j}}),\, 0 < j$;
\item[2.ii.] $\mini(H_{2}^{P_j}) \neq \mini(H_2^{P_{j-1}})$ and $\last(H_{2}^{P_0}) \neq \mini(H_2^{P_{j}}),\,0 < j$;
\end{enumerate}
Notice that, by $\A$ sharp arrangement, conditions 2.i. and 2.ii. can only fail for points $P_j$ such that $m(P_j)=2$.
\end{remark}

\begin{example}\label{example1}
Let $\A$ be a sharp arrangement such that condition (\ref{cond:last}) is satisfied for at least one choice of polar coordinate system $(V_0,V_1)$ among those that are described in Remark \ref{rem:diffpol}. Then with Theorem \ref{thm1.1} we have that:
\begin{center}
if $\gcd(3,m(P_0))=1,$ then $\A$ is a-monodromic.
\end{center}
Such arrangement is given in Example \ref{example4}.
\end{example}

Theorem \ref{thm1.1} provides a case where Conjecture \ref{conj:amon} holds. Indeed by the upper bound for torsion local system given by Papadima-Suciu (see \cite[Theorem C]{PS}) and the vanishing of the Aomoto complex when $\Gamma(\A)$ is connected (see \cite[Lemma 2.1]{Bailet}), fist author proved (see \cite{Bailet}) that Conjecture \ref{conj:amon} holds if $\A$ only admits monodromies $p^{k}$, $p$ prime.
\begin{corollary}
Let $\A$ in $\R^2$ be a sharp arrangement such that 
$$\last(H_{m(P_i)-1}^{P_i}) \neq \last(H_2^{P_j}), \quad 0 \leq i <j $$ 
holds for hyperplanes in $\A_{(0,3)}$. If the graph of double points $\Gamma(\A)$ is connected, then $\A$ is a-monodromic. 
\end{corollary}
\begin{remark}The reciprocal is false, as shown in example \ref{example0}. Indeed, $H_2^{P_4}$ does not contain any double point and $\Gamma(\A)$ is not connected. 
Remark also that $\A$ has no multinet, since $\A$ is a-monodromic. 
\end{remark}

The rest of this section is devoted to prove our main result \ref{thm1}, consequence of Lemma's \ref{lem1} and \ref{propmono3last}

\begin{lemma}\label{lem1}
Let $\A$ in  $\R^2$ be a sharp arrangement such that $$\last(H_{m(P_i)-1}^{P_i}) \neq \last(H_2^{P_j}), \quad 0 \leq i <j $$
holds for hyperplanes in $\A_{0}$. If $H_h^{P_i} \in \A_{0}$ is a vertex in a cycle of $\mathcal{G}_{\A_{0},\last}$ as the one in Figure \ref{fig:cycle}, then $h=2$ or $h=m(P_i)-1.$ More precisely we have the cycle in Figure \ref{fig:cycleA0}. 
\end{lemma}
\begin{figure}[htbp]
\centering
\begin{tikzpicture}

\tikzset{vertex/.style = {shape=rectangle}}
\tikzset{edge/.style = {->,> = latex'}}
\node[vertex] (h) at  (12,0) {$H_2^{P_j}$};
\node[vertex] (h-h1) at  (10.5,0.4) {$\pol$};
\node[vertex] (h1) at  (9,0) {$H_2^{P_{j_1}}$};
\node[vertex] (h1-h2) at  (7.5,0.4) {$\pol$};
\node[vertex] (h2) at  (6,0) {$H_{2}^{P_{j_2}}$};

\node[vertex] (pointshaut) at  (4.5,0) {$\hdots$};

\node[vertex] (hk) at  (3.5,0) {$H_{2}^{P_{j_k}}$};\node[vertex] (hk-hbar) at  (2.3,0.4) {$\pol$};
\node[vertex] (hbar) at (0,0) {$H_2^{P_i}$ or $H_{m(P_i)-1}^{P_i}$};

\node[vertex] (h'-hbar) at  (-1,-1) {$\nabla \,\text{or}\,\Delta$};
\node[vertex] (h') at  (0,-2) {$H_2^{P_{j'}}$};
\node[vertex] (h'-hk1) at  (1.5,-2.4) {$\rhd$};
\node[vertex] (hk1) at  (3,-2) {$H_{2}^{P_{j_{k+1}}}$};
\node[vertex] (hk1-hk2) at  (4.5,-2.4) {$\rhd$};
\node[vertex] (hk2) at  (6,-2) {$H_2^{P_{j_{k+2}}}$};

\node[vertex] (pointsbas) at  (7.5,-2) {$\hdots$};

\node[vertex] (hks) at  (9,-2) {$H_{2}^{P_{j_{k+s}}}$};
\node[vertex] (hks-htilde) at  (10.5,-2.4) {$\rhd$};
\node[vertex] (htilde) at  (12,-2) {$H_{m(P_{\widetilde{j}})-1}^{P_{\widetilde{j}}}$};

\node[vertex] (h-htilde) at  (12.4,-1) {$\nabla$};


\draw[edge] (hbar.east)  to[bend left] (hk.west);
\draw[edge] (h2.east)  to[bend left] (h1.west);
\draw[edge] (h1.east)  to[bend left] (h.west);
\draw[edge] (h'.north)  to[bend left] (hbar.south);
\draw[edge] (hk1.west)  to[bend left] (h'.east);
\draw[edge] (hk2.west)  to[bend left] (hk1.east);
\draw[edge] (htilde.west)  to[bend left] (hks.east);
\draw[edge] (h.south)  to[bend left] (htilde.north);

\end{tikzpicture}
\caption{Cycles in $\mathcal{G}_{\A_{0},\last}$ }
\label{fig:cycleA0}
\end{figure}
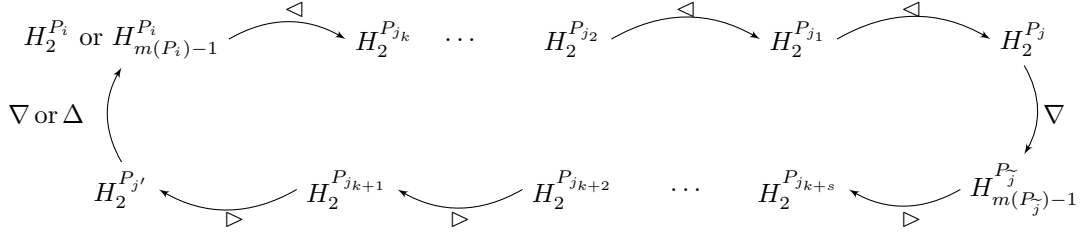
 
Note that Lemma \ref{lem1} essentially states that the only edges involved in a cycle are edges of the form $E_1$ for $k=2, m(P_i)-1$, and $E_2$ (see Figure \ref{fig:edge_glast}).

\proof Let us consider the cycle $\gamma$ in Figure \ref{fig:cycle}. By Remark \ref{rk1} we know that any edge $[H_h^{P_i},H_k^{P_j},\pol]$ have to be of the form $E_1$, that is $k=2$. Then all vertices in $\gamma$ but $\overline{H}$ and $\widetilde{H}$ are of the form $H_2^{P_j}.$ Let us now deal with $\overline{H}$ and $\widetilde{H}$ and their edges in $\gamma$. If $\overline{H} \rhd H',$ then $\overline{H}$ is also of the form $H_2^{P_j}.$ Assume $\overline{H} \pol H'$, then we are in the case $H'=H_2^{P_j}$, $\overline{H}=H_h^{P_i}$ with $i<j$, and one of the following two situations apply to the edge $[H',\overline{H},\rhd]$:
\begin{enumerate}
\item $\overline{H} \in S(\last(H'))$ and by Remark \ref{rk1} 2. $\overline{H}= H_{m(P_i)-1}^{P_i}$. Since, by Remark \ref{rk1} 4., $\overline{H}=H_{m(P_i)-1}^{P_i}, \,0<i<j$, implies $\last(\overline{H})= \last(H')$, it follows that $\overline{H}= H_{m(P_0)-1}^{P_0}$;
\item $\overline{H} \in \bold{\widehat{U}}_{Max}(\last(H'))$ and, by $i<j$, we are in case Remark \ref{rk1} 3. (b) i., that is $\overline{H}=H_{m(P_i)-1}^{P_i},\,0<i<j,$ with $H_{m(P_i)}^{P_i}\parallel H'$ or $H_{m(P_i)}^{P_i} \in S(\last(H')),$ (edge $E_5$). Indeed, case 3. (b) ii. $\overline{H}=H_{m(P_i)-2}^{P_i}$ is only possible if $H_{m(P_i)-1}^{P_i} \in S(\last(H'))$ that is again, by Remark \ref{rk1} 4.,  $\last(H_{m(P_i)-1}^{P_i})= \last(H'),\, i<j$. 
\end{enumerate}
Similarily, the existence of the edge $[H,\widetilde{H},\rhd]$ implies by Remark \ref{rk1} that $\widetilde{H}=H_{m(P_i)-1}^{P_i}$ for a certain $i,\, 0 \leq i \leq |\mathcal{P}|.$\\
By denoting $H=H_2^{P_j}$, $\overline{H}=H_2^{P_i}$ or $H_{m(P_i)-1}^{P_i}$, $H'=H_2^{P_{j'}}$, $\widetilde{H}=H_{m(P_{\widetilde{j}})-1}^{P_{\widetilde{j}}}$ and $H_{i_k}=H_2^{P_{j_k}}$, we get the cycle in Figure \ref{fig:cycleA0} from the one in Figure \ref{fig:cycle}, recalling that we denote vertices of $\mathcal{G}_{\A_0,\last}$ only by hyperplanes $H$ omitting points in $H$ as they are always of the form $\last(H)$.
\endproof

\begin{lemma}\label{propmono3last}
Let $\A$ in $\R^2$ be a sharp arrangement such that $$\last(H_{m(P_i)-1}^{P_i}) \neq \last(H_2^{P_j}), \quad 0 \leq i <j $$ 
holds for hyperplanes in $\A_{0}$. If there exists a cycle in $\mathcal{G}_{\A_0,\last}$ as in Figure \ref{fig:cycleA0}, then it involves at least a vertex $H_h^{P_i}$ with $m(P_i)=3,$ i.e. $\mathcal{G}_{\A_{(0,3)},\last}$ doesn't contain any cycle and it is an oriented forest.
\end{lemma}

\proof
Let us consider three hyperplanes $H_h^{P_j} \pol H_2^{P_i} \pol H_2^{P_k}$ in $\A_0$ with $0 \leq j< i<k$ connected as in Figure \ref{fig:cycleA0} by edges $[H_h^{P_j},H_2^{P_i},\pol],[H_2^{P_i},H_2^{P_k},\pol]$ that is, by Remark \ref{rk1}, $H_2^{P_k}\in S(\last(H_2^{P_i}))$ and $H_2^{P_i}\in S(\last(H_h^{P_j})).$ It is an easy geometric remark that, on sharp arrangements, this configuration forces $\last(H_2^{P_i})=\last(H_h^{P_j})$ and it follows by Propositions \ref{prop2} and by our assumptions on $\last$ that $h=2$ and $m(P_j)=3$. \\
Hence the only cases left are the cycles with no $2$ consecutive edges of the form $[H'',H', \pol],[H',H,\pol]$, that is, by Lemma \ref{lem1} (see Figure \ref{fig:cycleA0}), our cycle is composed by two edges: $$[H,H',\rhd] \mbox{ and }[H',H,\pol] \quad \mbox{(case A)};$$ 
or four edges: $$[H,\widetilde{H},\rhd],[\widetilde{H},H',\pol],[H',\overline{H},\rhd],[\overline{H},H,\pol] \quad \mbox{(case B)};$$ 
or three edges:
$$[H,\widetilde{H},\rhd],[\widetilde{H},H'],[H',H,\pol],$$
where $\widetilde{H}=H_{m(P_{\widetilde{j}})-1}^{P_{\widetilde{j}}}=H_2^{P_{j'}}$ in Figure \ref{fig:cycleA0}, that is $P_{j'}=P_{\widetilde{j}}$ has multiplicity 3.

Note that in the above cases $H=H_2^{P_j}$. Let us study cases A and B separately:
\begin{enumerate}
\item[A.] $H \in S(\last(H'))$ that is, by Remark \ref{rk1}, $H'=H_{m(P_0)-1}^{P_0} \in S(\last(H))$ or $H'=H_{m(P_i)-1}^{P_i} \in \bold{\widehat{U}}_{Max}(\last(H)),\,0<i<j.$ The first configuration is an absurd since in this case $\last(H)=\last(H').$ In the second configuration, the fact that $H'\in \Cone(\last(H))$ implies the existence of a line $H'' \pol H'$ such that $H'' \in S(\last(H)).$ It is obvious that $H''$ has to be parallel with $H'$ as $H \in \last(H')$, that is $m(P_i)-1=2$ and $m(P_i)=3.$
\item[B.] This case corresponds to the cycle in Figure \ref{fig:cycle4} and we have the following two cases:
\begin{figure}[htbp]
\centering
\begin{tikzpicture}

\tikzset{vertex/.style = {shape=rectangle}}
\tikzset{edge/.style = {->,> = latex'}}
\node[vertex] (h1) at  (9,0) {$H=H_2^{P_j}$};
\node[vertex] (h1-h2) at  (7.5,0.4) {$\pol$};
\node[vertex] (h2) at  (6,0) {$\overline{H}=H_{m(P_i)-1}^{P_i}$};

\node[vertex] (h'-hbar) at  (5,-1) {$\Delta$};
\node[vertex] (hk2) at  (6,-2) {$H'=H_2^{P_{j'}}$};
\node[vertex] (hks) at  (9,-2) {$\widetilde{H}=H_{m(P_{\widetilde{j}})-1}^{P_{\widetilde{j}}}$};
\node[vertex] (hk2-hks) at  (7.5,-2.4) {$\rhd$};
\node[vertex] (h-htilde) at  (9.4,-1) {$\nabla$};


\draw[edge] (h2.east)  to[bend left] (h1.west);
\draw[edge] (hk2.north)  to[bend left] (h2.south);
\draw[edge] (hks.west)  to[bend left] (hk2.east);
\draw[edge] (h1.south)  to[bend left] (hks.north);

\end{tikzpicture}
\caption{Case B}\label{fig:cycle4}
\end{figure}

\begin{enumerate}
\item $\overline{H}=H_{m(P_0)-1}^{P_0} \in S(\last(H')).$ Then the egde $[\overline{H},H,\pol]$ means that $H\in S(\last(\overline{H})).$ By assumption we have that $H \rhd H'.$ Since $\last(\overline{H}) \rhd_= \last(H'),$ we easily see that either $\last(H')\in H$ and $\last(\overline{H})=\last(H'),$ not possible by hypothesis, or $H\parallel H'$ and $m(H' \cap H_1^{P_0})=2$, i.e. $H' \notin \A_0.$ 
\item $\overline{H}=H_{m(P_i)-1}^{P_i} \in \bold{\widehat{U}}_{Max}(\last(H')),\,0<i<j.$ The fact that $\overline{H}\in \Cone(\last(H'))$ implies the existence of a line $H'' \pol \overline{H}$ such that $H'' \in S(\last(H')).$ If $H'' \parallel \overline{H},$ then $m(P_i)-1=2$ and $m(P_i)=3.$ Otherwise, $H'' \cap \overline{H} \neq \emptyset$ and $\last(\overline{H}) \rhd_= H'' \cap \overline{H}.$ On the other hand, the edge $[\overline{H},H,\pol]$ means that $H\in S(\last(\overline{H}))$ and since $H' \pol H$ we easily see that either $H \parallel H'$ and $H' \notin \A_0$ or $H \cap H' \rhd \last(H')$ along $H',$ wich is an absurd.
\end{enumerate}
\end{enumerate}
\endproof

Let us remark that in Proposition \ref{thm1} and Theorem \ref{thm1.1} we focused on the arrangement $\A_{(0,3)} \subset \A_0$ since our main goal is to show Conjecture \ref{conj:mon}. But above Lemma's are true more in general if condition (\ref{cond:last}) holds in $\A_0$. Hence our algorithm also provides a way to show a-monodromicity directly via the following result.
\begin{proposition}\label{thm0}Let $\A$ in $\R^2$ be a sharp arrangement such that
$$\last(H_{m(P_i)-1}^{P_i}) \neq \last(H_2^{P_j}), \quad 0 \leq i <j $$
holds for hyperplanes in $\A_0$ and $\mathcal{G}_{\A_0,\last}$ doesn't have any cycle as the one in Figure \ref{fig:cycleA0}, then $\mathcal{G}_{\A_{0},\last}$ is an oriented forest and $\A$ is a-monodromic.
\end{proposition}

\begin{example}\label{example0}
Let us consider the sharp arrangement $\A$ in $\R^2$ depicted in Figure \ref{fig:example0}.
The cardinality of $\overline{\A}$ is 12 and $\A_0=\{H_2^{P_0},H_2^{P_4}\},$ since $m(P_1)=m(P_2)=m(P_3)=m(P_5)=2$ and $m(P_6)=4$ is coprime with $m(P_0)=3.$ Since $\last(H_2^{P_0}) \neq \last(H_2^{P_4})$ and the graph $\mathcal{G}_{\A_{0},\last}$ is composed of two non connected vertices, we have that $\A$ is a-monodromic by Proposition \ref{thm0}.

\end{example}

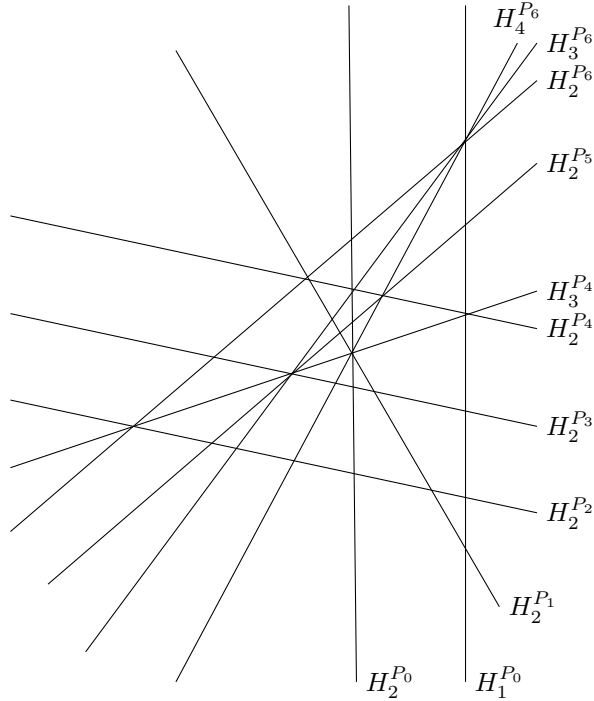
\begin{figure}[htbp]
\centering
\begin{tikzpicture}[scale=1]

\draw (0.05,0) node [right] {$H_1^{P_0}$} --(0.05,9);
\draw (-1.4,0) node [right] {$H_2^{P_0}$} --(-1.5,9);

\draw (1,2.25) node [right] {$H_2^{P_2}$} --(-6,3.75);

\draw (1,3.4) node [right] {$H_2^{P_3}$} --(-6,4.9);

\draw (1,4.7) node [right] {$H_2^{P_4}$} --(-6,6.2);
\draw (1,5.2) node [right] {$H_3^{P_4}$} --(-6,2.85);

\draw (1,6.9) node [right] {$H_2^{P_5}$} --(-5.5,1.3);

\draw (1,8) node [right] {$H_2^{P_6}$} --(-6,2);
\draw (1,8.5) node [right] {$H_3^{P_6}$} --(-5,0.4);
\draw (0.74,8.5) node [above] {$H_4^{P_6}$} --(-3.8,0);

\draw (0.5,1) node [right] {$H_2^{P_1}$} --(-3.8,8.4);

\end{tikzpicture}
\caption{$\A$ is a-monodromic}
\label{fig:example0}
\end{figure}

Let us now study the case in which we have $\last(H_{m(P_i)-1}^{P_i}) = \last(H_2^{P_j}),$ for certain $0 \leq i < j.$ We will consider a different subgraph: $\mathcal{G}_{\A'_0,\last, \mini}$ involving different vertices.

\subsection{The subgraph $\mathcal{G}_{\A'_0,\last, \mini}$}\label{subgraph2}
 
 Analogously to the case of graph $\mathcal{G}_{\A_0,\last}$, as a first step we will simplify rows $H_k^{P_0}$ with the following Lemma. 

\begin{lemma}Rows $\{H^{P_0}_k\}_{2< k <m(P_0)-1}$ can be removed in the sense of Definition \ref{def:remove} without changing columns $c_H^{\last(H)}$, $H \in \A' \setminus \{H_2^{P_i}\}_{0 \leq i \leq |\mathcal{P}|}$, and columns $c_H^{\mini(H)}$, $H \in \{H_2^{P_i}\}_{0 \leq i \leq |\mathcal{P}|}$.
\end{lemma}

\proof In order to simplify rows $H_k^{P_0},\, 2<k<m(P_0)-1,$ in the matrix $D(M),$ let us consider columns $c_{H_k^{P_0}}^{\last(H_k^{P_0})}$ defined in Notation \ref{not:goodcolumn}. It is clear that $e(H',c_{H_k^{P_0}}^{\last(H_k^{P_0})})=0$ for all $H'\pol H_k^{P_0}.$ In order to remove the entries $e(H',c_{H_k^{P_0}}^{\last(H_k^{P_0})})\neq 0$ with $H'\rhd H_k^{P_0},$ we perform usual rows operations by using the entry $e(H_k^{P_0},c_{H_k^{P_0}}^{\last(H_k^{P_0})})=t^\alpha(1-t).$ By Remarks \ref{simplification1} and \ref{rksimplast} these row operations do not affect the other columns $c_H^{\last(H)},\,H\neq H_2^{P_i}$, and $c_{H_2^{P_i}}^{\mini(H_2^{P_i})}$ of $D(M).$

Finally, we have that the last points $\last(H_k^{P_0}),\, 2<k<m(P_0)-1,$ are different from the $\last(H),\,H\neq H_2^{P_i},$ and the $\mini(H_2^{P_i})$ (see Remark \ref{lastmininfty} 1.).
\endproof

For the sake of simplicity, in the rest of this section we will assume $m(P_0) > 3$. This choice is essentially due to avoid case $H_{m(P_0)-1}=H_2^{P_0}$ in which additional considerations would be necessary in order to decide wether to consider the last or the min point along this line.
This condition is not strong one from our point of view as our main goal is to prove Conjecture \ref{conj:mon}. Note that now the line $H_2^{P_0}$ cannot be removed anymore hence we will deal with the arrangement $\A'_0$ defined in Notation \ref{not:arrnozero}.   

\begin{definition}\label{def:grafGLM}
We define $\mathcal{G}_{\A'_0,\last, \mini}$ the subgraph of $\mathcal{G}_{\A'}$ such that vertices are of the form $(H_2^{P_i},\mini(H_2^{P_i}))$, $ 0 \leq i \leq |\mathcal{P}|$, and  $(H,\last(H))$ if $H \neq H_2^{P_i}$, $H \in \A'_0$.
\end{definition}

By definition of $\mathcal{G}_{\A'_0,\last, \mini}$ it follows that its subgraph involving vertices $(H,\last(H))$ if $H \neq H_2^{P_i}$, $H \in \A'_0$, is a subgraph of $\mathcal{G}_{\A_0,\last}$. Then we need only to study edges connecting the new vertices $(H_2^{P_i},\mini(H_2^{P_i}))$, $ 0 \leq i \leq |\mathcal{P}|$. 

\begin{remark}\label{rklasth2}
Let $H_k^{P_j} \in \A'$ such that $k \neq 2.$ Then the following are easy geometric remarks on sharp arrangements:
\begin{enumerate}
\item if $H_2^{P_0}\in S(\last(H_k^{P_j})),$ then $m(P_0)=3$;
\item $H_2^{P_0}\in \bold{\widehat{U}}(\last(H_k^{P_j})) \Leftrightarrow H_{m(P_0)-1}^{P_0} \in S(\last(H_k^{P_j})).$   
\end{enumerate}
If $H_2^{P_0} \neq H_{m(P_0)-1}^{P_0}$, i.e. $m(P_0) >3$, by Definition \ref{def:Uppmax} of $\bold{\widehat{U}}_{Max}(\last(H_k^{P_j}))$ it follows: 
$$H_2^{P_0} \notin S(\last(H_k^{P_j})) \cup \bold{\widehat{U}}_{Max}(\last(H_k^{P_j})) \cup \mathcal{N}(\last(H_k^{P_j})), \mbox{ for all } H_k^{P_j} \in \A'_0 .$$ 
\end{remark}

\begin{remark}\label{rkmin}With usual notations, for points $\mini(H), H \in \A'$, we have the following facts, consequence of $\A$ being a sharp arrangment:
\begin{enumerate}
\item if $H_h^{P_i} \in S(\mini(H_2^{P_j}))$ then 
$\left \{ \begin{array}{cc} 
h=2 & \mbox{ if } i<j \\
h=m(P_i) & \mbox{ if } i>j \\
\end{array}
\right .$\\
that is, as $H_{m(P_i)}^{P_i} \notin \A'$, this corresponds to edges  $E_7=[H_2^{P_j},H_2^{P_i},\rhd]$ (i.e. $j>i \geq 0$). 
\item If $H_h^{P_i}\in \bold{\widehat{U}}_{Max}(\mini(H_2^{P_j}))$ then $j<i$ and $h=m(P_i)$, i.e. $H_{h}^{P_i} \notin \A'$,  or $h=m(P_i)-1$ and $H_{m(P_i)}^{P_i}\in S(\mini(H_2^{P_j})).$ The corresponding edges are $E_{8}=[H_2^{P_j},H_{m(P_i)-1}^{P_i},\pol]$ (i.e. $0 \leq j <i$ ).
\end{enumerate}
\end{remark}

\begin{figure}[htbp]
\centering
\begin{tikzpicture}

\tikzset{vertex/.style = {shape=rectangle}}
\tikzset{edge/.style = {->,> = latex'}}
\node[vertex] (h) at  (12.5,0) {$H_{k-1}^{P_j}$};
\node[vertex] (h-h1) at  (11,0.6) {$\rhd$};
\node[vertex] (h1) at  (9.5,0) {$H_k^{P_{j}}$};
\node[vertex] (h11-h2) at  (6,0.6) {$\rhd$};
\node[vertex] (h2) at  (4.5,0) {$H_{2}^{P_{j}}$};
\node[vertex] (h11) at  (7.8,0) {$H_{2}^{P_{i}}$};

\node[vertex] (hk) at  (3,0) {$H_{2}^{P_{i}}$};
\node[vertex] (hk-hbar) at  (1.5,0.6) {$\pol$};
\node[vertex] (hbar) at (0,0) {$H_k^{P_j}$ };

\node[vertex] (hk1) at  (4.5,-2) {$H_{k}^{P_j}$};
\node[vertex] (hk1-hk2) at  (6,-1.5) {$\rhd$};
\node[vertex] (hk2) at  (7.8,-2) {$H_{m(P_i)-1}^{P_{i}}$};


\node[vertex] (hks1) at  (0,-2) {$H_{2}^{P_{j}}$};
\node[vertex] (hks1-htilde1) at  (1.3,-1.5) {$\pol$};
\node[vertex] (htilde1) at  (3,-2) {$H_{m(P_i)-1}^{P_{i}}$};



\draw[edge] (hbar.east)  to[bend left] (hk.west);
\draw[edge] (h2.east)  to[bend left] (h11.west);
\draw[edge] (h1.east)  to[bend left] (h.west);
\draw[edge] (hk1.east)  to[bend left] (hk2.west);
\draw[edge] (hks1.east)  to[bend left] (htilde1.west);


\node[vertex] (hk-hbar) at  (1.5,-0.5) {\scriptsize{$E_1: 0 \leq j < i, k \neq 2$}};
\node[vertex] (h11-h2) at  (6,-0.5) {\scriptsize{$E_7: 0 \leq i < j$}};
\node[vertex] (h-h1) at  (11,-0.5) {\scriptsize{$E_4: 0 < j$}};
\node[vertex] (hk1-hk2) at  (6,-2.7) {\scriptsize{$E_3 \mbox{ and } E_5: 0 \leq i < j, k \neq 2$}};
\node[vertex] (hks1-htilde1) at  (1.5,-2.7) {\scriptsize{$E_8: 0 \leq j < i$}};

\end{tikzpicture}
\caption{Edges in $\mathcal{G}_{\A'_0,\last, \mini}$ }
\label{fig:edge_glmin}
\end{figure}

By previous remarks and by Remark \ref{rk1} we have that all possible edges in
$\mathcal{G}_{\A'_0,\last, \mini}$ are the one depicted in Figure \ref{fig:edge_glmin}. Let us remark that, with respect the subgraph $\mathcal{G}_{\A_0,\last}$, new edges of the form $E_7$ and $E_8$ appeared (see Figure \ref{fig:edge_glast}) while $E_2$ and $E_6$ disappeared. The latter follows from the fact that smaller lines of type $H_{m(P_i)-1}^{P_i}$ and $H_{m(P_i)-2}^{P_i}$ cannot be either in the $S(\mini(H_2^{P_j}))$ nor in  $\bold{\widehat{U}}_{Max}(\mini(H_2^{P_j}))$ for $j>i$ for obvious geometric reasons as $\A$ is sharp. For the same reason, type $E_1$ and $E_3,E_5$ edges require condition $k \neq 2$ (otherwise we would have points of multiplicity 2 or 3). We also have the following two easy facts: 

\begin{enumerate}
\item if $\last(H_h^{P_i})=\mini(H_2^{P_j}),$ then $H_h^{P_i}\in S(\mini(H_2^{P_j}))$ and by Remark \ref{rkmin} we have that $h=2$ and $i<j$, the case $h=m(P_i)$ being impossible since $H_h^{P_i}\in \A'$. In particular, if $\last(H_{m(P_0)-1}^{P_0})=\mini(H_2^{P_j}),$ then $m(P_0)=3;$
\item if $\mini(H_2^{P_i})=\mini(H_2^{P_j}),$ then 
\begin{enumerate}
\item $i<j \Rightarrow m(P_j)=2$ and $H_2^{P_j}\notin \A'$,
\item $j<i \Rightarrow m(P_i)=2$ and $H_2^{P_i}\notin \A'$;
\end{enumerate}
\end{enumerate}

In the rest of this subsection, Theorem, Corollary, Lemma and Proposition analogous to the one in previous subsection are stated and proved. Define
\begin{equation}\label{eq:A034}
\A'_{(0,3,4)}= \A'_0 \backslash \{H_h^{P_i}\,|\,P_i\in \mathcal{P} \,\text{and}\,m(P_i)\in\{3,4\}\},
\end{equation}
then the following result holds. 

\begin{proposition}\label{thm2}
Let $\A$ in $\R^2$ be a sharp arrangement. If $\mathcal{G}_{\A'_{(0,3,4)},\last,\mini}$ does not contain any cycle of length $l$ as the one in Figure \ref{fig:cycle} such that
\begin{itemize}
\item $\overline{H}\rhd H'$ and $l$ is odd
\item $\overline{H}\pol H'$ and $l$ is even
\end{itemize} 
then the matrix associated to the graph $\mathcal{G}_{\A'_{(0,3,4)},\last,\mini}$ can be diagonalized in $\diag(1-t).$
\end{proposition}

The following main Theorem, stated in Introduction, follows from previous Proposition, by Remarks \ref{lastmininfty} and \ref{rkamono}.

\begin{theorem}\label{thm2.1}
 Let $\A$ in $\R^2$ be a sharp arrangement. If $\mathcal{G}_{\A'_{(0,3,4)},\last,\mini}$ does not contain any cycle of length $l$ as the one in Figure \ref{fig:cycle} such that:
\begin{itemize}
\item $\overline{H}\rhd H'$ and $l$ is odd
\item $\overline{H}\pol H'$ and $l$ is even
\end{itemize} 
then $\A$ is a-, 3- or 4-monodromic.
\end{theorem}

When the assumptions of Theorem \ref{thm2.1} are satisfied, we can deduce, as in previous section, the a-monodromicity of $\A$ from the connectivity of $\Gamma(\A).$ 
\begin{corollary}
 Let $\A$ in $\R^2$ be a sharp arrangement. Assume $\mathcal{G}_{\A'_{(0,3,4)},\last,\mini}$ does not contain any cycle of length $l$ as the one in Figure \ref{fig:cycle} such that:
\begin{itemize}
\item $\overline{H}\rhd H'$ if $l$ is odd
\item $\overline{H}\pol H'$ if $l$ is even,
\end{itemize} 
then, if $\Gamma(\A)$ is connected, $\A$ is a-monodromic.
\end{corollary}

The rest of this section is devoted to prove our main result \ref{thm2}, consequence of Lemma's \ref{lem2} and \ref{prop4}.

\begin{lemma}\label{lem2}
Let $\A$ in $\R^2$ be a sharp arrangement. If $H_h^{P_i} \in \A_0'$ is a vertex in a cycle of $\mathcal{G}_{\A'_0,\last, \mini}$ as the one in Figure \ref{fig:cycle}, then $h=2 \,\text{or}\,\, m(P_i)-1\,\text{or}\,\,m(P_i)-2,$ the latter being only possible if  $i>0.$ More precisely we have cycles as the one in Figure \ref{fig:cyclastmin}.
\end{lemma}
\begin{figure}[htbp]
\centering
\begin{tikzpicture}

\tikzset{vertex/.style = {shape=rectangle}}
\tikzset{edge/.style = {->,> = latex'}}
\node[vertex] (h) at  (12,0) {$H=H_2^{P_{j_0}} (H_{m_{j_0}}^{P_{j_0}})$};
\node[vertex] (h-h1) at  (10,0.4) {$\pol$};
\node[vertex] (h1) at  (8.5,0) {$H_{m_{j_1}}^{P_{j_1}}(H_{2}^{P_{j_1}})$};
\node[vertex] (h1-pointshaut) at  (6.8,0.4) {$\pol$};
\node[vertex] (pointshaut) at  (5.5,0) {$\hdots$};

\node[vertex] (hk) at  (3.7,0) {$H_{m_{j_k}}^{P_{j_k}}$ or $H_2^{P_{j_k}}$};
\node[vertex] (hk-hbar) at  (1.5,0.4) {$\pol$};
\node[vertex] (hbar) at (0,0) {$\overline{H}=H_{\overline{j}}^{P_i}$};

\node[vertex] (h'-hbar) at  (-1,-1) {$\nabla \mbox{ or } \Delta$};
\node[vertex] (h') at  (0,-2) {$H'=H_{m_{j'}}^{P_{j'}}$ or $H_2^{P_{j'}}$};
\node[vertex] (h'-hk1) at  (2.3,-2.5) {$\rhd$};

\node[vertex] (pointsbas) at  (3.5,-2) {$\hdots$};
\node[vertex] (hks-htilde) at  (4.6,-2.5) {$\rhd$};
\node[vertex] (hks) at  (6.8,-2) {$H_{m_{j_{k+s}}}^{P_{j_{k+s}}}(H_2^{P_{j_{k+s}}})$};
\node[vertex] (hks-htilde) at  (9,-2.5) {$\rhd$};
\node[vertex] (htilde) at  (12,-2) {$\widetilde{H}=H_{2}^{P_j}(H_{m_j}^{P_j} \mbox { or } H_{m_j-1}^{P_j})$};
\node[vertex] (h-htilde) at  (12.4,-1) {$\nabla$};


\draw[edge] (hbar.east)  to[bend left] (hk.west);
\draw[edge] (pointshaut.east)  to[bend left] (h1.west);
\draw[edge] (h1.east)  to[bend left] (h.west);
\draw[edge] (h'.north)  to[bend left] (hbar.south);
\draw[edge] (pointsbas.west)  to[bend left] (h'.east);
\draw[edge] (hks.west)  to[bend left] (pointsbas.east);
\draw[edge] (htilde.west)  to[bend left] (hks.east);
\draw[edge] (h.south)  to[bend left] (htilde.north);

\end{tikzpicture}
\caption{Cycles in $\mathcal{G}_{\A'_0,\last, \mini}$ with $m_{j_h}=m(P_{j_h})-1$, the vertices alternating and $\overline{j}=2, m_{\overline{j}} , m_{\overline{j}}-1$.}
\label{fig:cyclastmin}
\end{figure}
\proof
By Remarks \ref{rk1} and \ref{rkmin} (see also Figure \ref{fig:edge_glmin}), we know that edges $[H,H',\pol]$ in $\mathcal{G}_{\A'_0,\last,\mini}$ are of the form $E_1$ and $E_{8}$. It follows that vertices $H, H' $ and $H_{i_j}$, $ 1 \leq j \leq k+s$ have to be of type $H_2^{P_i}$ or $H_{m(P_i)-1}^{P_i}$ and, moreover, the type alternates if $m(P_i)\neq3$, i.e. two adjacent vertices have to be of different type.
By Remarks \ref{rk1} and \ref{rkmin} one can asses also the exact values of $\overline{H}$ and $\widetilde{H}$ depending on the types of $H$, $H'$ and the sign of the edge $[H,H']$ (see Figure \ref{fig:edge_glmin}). More precisely we have:
\begin{enumerate}
\item if $H=H_2^{P_{j_0}}$ then $\widetilde{H}=H_2^{P_j}$ (see $E_7$) corresponding, respectively, to type $H_{m(P_{j_{k+s}})-1}^{P_{j_{k+s}}}$ for the subsequent vertex in the cycle in Figure \ref{fig:cyclastmin} (see $E_8$);
\item if $H=H_{m(P_{j_0})-1}^{P_{j_0}}$ then $\widetilde{H}=H_{m(P_{j_0})-2}^{P_{j_0}}$ (see $E_4$) or $\widetilde{H}=H_{m(P_j)-1}^{P_j}$ (see $E_3$ and $E_5$) corresponding both to type $H_2^{P_{j_{k+s}}}$ for the subsequent vertex in the cycle in Figure \ref{fig:cyclastmin} (see $E_1$);
\item if $H' \rhd \overline{H}$ then as $H'=H_2^{P_{j'}}$ or $H'=H_{m(P_{j'})-1}^{P_{j'}}$ previous points 1. and 2. respectively apply with $H'$ instead of $H$ and $\overline{H}$ instead of $\widetilde{H}$;
\item if $H' \pol \overline{H}$ then same alternating rule of other vertices applies (see $E_1$ and $E_8$), that is if 
$H'=H_2^{P_{j'}}$ ($H_{m(P_{j'})-1}^{P_{j'}}$) then $\overline{H}=H_{m(P_{i})-1}^{P_{i}}$ ($H_2^{P_{i}}$)
\end{enumerate}
and we get cycles as the one in Figure \ref{fig:cyclastmin}. 
\endproof

\begin{figure}[htbp]
\centering
\begin{tikzpicture}

\tikzset{vertex/.style = {shape=rectangle}}
\tikzset{edge/.style = {->,> = latex'}}
\node[vertex] (hbar) at (0,0) {$\overline{H}=H_{\overline{j}}^{P_i}$};

\node[vertex] (h) at  (12,0) {$H=H_2^{P_{j_0}} (H_{m_{j_0}}^{P_{j_0}})$};
\node[vertex] (h-h1) at  (10,0.4) {$\pol$};
\node[vertex] (h1) at  (8.5,0) {$H_{m_{j_1}}^{P_{j_1}}(H_{2}^{P_{j_1}})$};
\node[vertex] (h1-pointshaut) at  (6.8,0.4) {$\pol$};
\node[vertex] (pointshaut) at  (5.5,0) {$\hdots$};
\node[vertex] (hk) at  (3.7,0) {$H_{m_{j_k}}^{P_{j_k}}$ or $H_2^{P_{j_k}}$};
\node[vertex] (gm) at  (0,1) {$\gamma_1 :$};
\node[vertex] (gm2) at  (0,-1.2) {$\gamma_2 :$};
\node[vertex] (h') at  (0,-2) {$H'=H_{m_{j'}}^{P_{j'}}$ or $H_2^{P_{j'}}$};
\node[vertex] (h'-hk1) at  (2.3,-2.5) {$\rhd$};
\node[vertex] (pointsbas) at  (3.5,-2) {$\hdots$};
\node[vertex] (hks-htilde) at  (4.6,-2.5) {$\rhd$};
\node[vertex] (hks) at  (6.8,-2) {$H_{m_{j_{k+s}}}^{P_{j_{k+s}}}(H_2^{P_{j_{k+s}}})$};
\node[vertex] (hks-htilde) at  (9,-2.5) {$\rhd$};
\node[vertex] (htilde) at  (12,-2) {$\widetilde{H}=H_{2}^{P_j}(H_{m_j}^{P_j} \mbox { or } H_{m_j-1}^{P_j})$};


\draw[edge] (hbar.east)  to[bend left] (hk.west);
\draw[edge] (pointshaut.east)  to[bend left] (h1.west);
\draw[edge] (h1.east)  to[bend left] (h.west);
\draw[edge] (pointsbas.west)  to[bend left] (h'.east);
\draw[edge] (hks.west)  to[bend left] (pointsbas.east);
\draw[edge] (htilde.west)  to[bend left] (hks.east);

\end{tikzpicture}
\caption{Subgraphs $\gamma_1$ and $\gamma_2$ of cycle $\gamma$ in $\mathcal{G}_{\A'_0,\last, \mini}$}
\label{fig:gammas}
\end{figure}
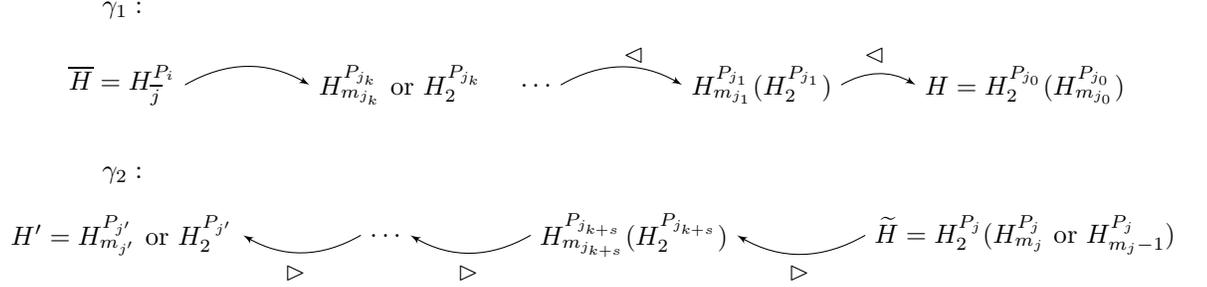

\begin{remark}\label{rem:lami}Let us remark that in the cycle $\gamma$ in Figure \ref{fig:cyclastmin} the type of $H'$ and $H_{h}^{P_{j_k}}$ ($h=m_{j_k}=m(P_{j_k})-1$ or $h=2$) depends, respectively, on the length of the paths $\gamma_2$ and $\gamma_1$ in Figure \ref{fig:gammas}. In particular if $H=H_2^{P_{j_0}} (H_{m(P_{j_0})-1}^{P_{j_0}})$ and $l_i$ is the length of $\gamma_i$,  we have:
\begin{enumerate} 
\item if $l_2=2h$ then $H'=H_2^{P_{j'}} (H_{m(P_{j'})-1}^{P_{j'}})$ and if $l_1=2h$ then $H_{\bullet}^{P_{j_k}}=H_{m(P_{j_k})-1}^{P_{j_k}}(H_2^{P_{j_k}})$;
\item if $l_2=2h-1$ then $H'=H_{m(P_{j'})-1}^{P_{j'}}(H_2^{P_{j'}})$ and if $l_1=2h-1$ then $H_{\bullet}^{P_{j_k}}=H_2^{P_{j_k}}(H_{m(P_{j_k})-1}^{P_{j_k}})$.
\end{enumerate}
By Lemma \ref{lem2} we also know that
\begin{enumerate}
\item[3.] if $H' \rhd \overline{H}$, $H'=H_2^{P_{j'}}(H_{m(P_{j'})-1}^{P_{j'}})$, then $\overline{H}=H_2^{P_{i}}(H_{m(P_{i})-1}^{P_{i}}$ or $H_{m(P_{i})-2}^{P_{i}})$;
\item[4.] if $H' \pol \overline{H}$ then same alternating rule of other vertices applies, that is if 
$H'=H_2^{P_{j'}}(H_{m(P_{j'})-1}^{P_{j'}})$ then $\overline{H}=H_{m(P_{i})-1}^{P_{i}} (H_2^{P_{i}})$.
\end{enumerate}
Finally remark that, by construction, if $\gamma$, $\gamma_1$ and $\gamma_2$ have lengths, respectively, $l, l_1$ and $l_2$ then $l=l_1+l_2+2$, that is if $l$ is even then $l_1$ and $l_2$ have to be both even or odd, while if $l$ is odd $l_1$ and $l_2$ have to be one odd the other even. 
\end{remark}
We can now prove the following final result.

\begin{lemma}\label{prop4}
Let $\A$ in $\R^2$ be a sharp arrangement, $\gamma$ a cycle in $\mathcal{G}_{\A'_0,\last, \mini}$ of length $l$. If
\begin{itemize} 
\item[i)]$\overline{H}\rhd H'$ and $l=2h$ or 
\item[ii)]$\overline{H}\pol H'$ and $l=2h+1$ 
\end{itemize}
then $\gamma$ contains a vertex $H_{\bullet}^{P_i}$ such that $m(P_i)\in \{3,4\}.$
\end{lemma}

\proof By Lemma \ref{lem2} and Remark \ref{rem:lami} we know vertices involved in the cycle $\gamma$ (see Figure \ref{fig:cyclastmin}). Let $l_{(\rhd,<)}$ be the number of edges in $\gamma$ with direction and sign opposite (as in Figure \ref{fig:use}).
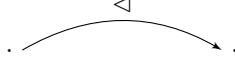
\begin{figure}[htbp]
\centering
\begin{tikzpicture}

\tikzset{vertex/.style = {shape=rectangle,minimum size=0.1em}}
\tikzset{edge/.style = {->,> = latex'}}
\node[vertex] (a) at  (0,0) {$.$};
\node[vertex] (b) at  (3,0) {$.$};

\node[vertex] (c) at  (1.5,0.6) {$\pol$};

\draw[edge] (a.east)  to[bend left] (b.west);
\end{tikzpicture}
\caption{\footnotesize{Edges with direction and sign opposite}}\label{fig:use}
\end{figure}
Remark that $l=l_{(\rhd,<)}+1,$ if $\overline{H}\rhd H'$ and $l=l_{(\rhd,<)}+2,$ if $\overline{H}\pol H'$. 
\begin{enumerate}
\item[i)] Case $\overline{H}\rhd H'$ and $l=2h$, that is $l_{(\rhd,<)}=2h-1$. If $H=H_{2}^{P_{j_0}}(H_{m(P_{j_0})-1}^{P_{j_0}})$ then, by the alternating rule of edges, going backwards along the cycle, $\widetilde{H}=H_{m(P_j)-1}^{P_j}(H_{2}^{P_j})$ (as $l_{(\rhd,<)}$ is odd), but on the other hand, considering that $[H,\widetilde{H}]$ is type $E_7$ ($E_4, E_5$ or $E_3$), we have $\widetilde{H}=H_{2}^{P_j} (H_{m(P_j)-2}^{P_j}$ or  $\widetilde{H}= H_{m(P_j)-1}^{P_j})$ (see Figure \ref{fig:cyclastmin}). Hence $m(P_j)=3$ or $4$.\\
\item[ii)] Case $l=2h+1$ and $\overline{H}\pol H'$. In this case we have that the cycle $\gamma$ is divided in the two disconnected subgraphs $\gamma_1$ and $\gamma_2$ removing edges $[\overline{H},H',\pol]$ and $[H,\widetilde{H},\pol]$ (see Figure \ref{fig:gammas}). By Remark \ref{rem:lami} since $l$ is even, if the lenght of $\gamma_1$ is even then the lenght of $\gamma_2$ has to be odd and viceversa. Assume that $l_1=2h_1-1$ and $l_2=2h_2$, the converse is similar. By Remark \ref{rem:lami} 1. and 2. if $H=H_2^{P_{j_0}}(H_{m(P_{j_0})-1}^{P_{j_0}})$ then $l_1=2h_1-1$ implies $H_{\bullet}^{P_{j_k}}=H_2^{P_{j_k}}(H_{m(P_{j_k})-1}^{P_{j_k}})$ and $l_2=2h_2$ implies $H'=H_2^{P_{j'}}(H_{m(P_{j'})-1}^{P_{j'}})$. The latter implies, by Remark \ref{rem:lami} 3, $\overline{H}=H_2^{P_{i}}(H_{m(P_{i})-1}^{P_{i}}$ or $H_{m(P_{i})-2}^{P_{i}})$. Then, since all possible edges in $\gamma$ are the one stated in Figure \ref{fig:edge_glmin}, it follows that $m(P_{j_k})=3$ ( $m(P_{j_k})=3$ or $m(P_i)=3,4$).
\end{enumerate}

\endproof

Let us remark that, analogously to the previous section, in Proposition \ref{thm2} and Theorem \ref{thm2.1} we focused on the arrangement $\A_{(0,3,4)} \subset \A_0$, since our main goal is to show Conjecture \ref{conj:mon}. But above Lemma's are true more in general in $\A'_0$. Hence our algorithm also provides a way to show a-monodromicity directly via the following result.
\begin{proposition}\label{thm0.2}Let $\A$ in $\R^2$ be a sharp arrangement. If $\mathcal{G}_{\A'_0,\last,\mini}$ does not contain any cycle of length $l$ as the one in Figure \ref{fig:cyclastmin} such that:
\begin{itemize}
\item $\overline{H}\rhd H'$ if $l$ is odd
\item $\overline{H}\pol H'$ if $l$ is even,
\end{itemize} 
then $\mathcal{G}_{\A'_{0},\last,\mini}$ is an oriented forest and $\A$ is a-monodromic.
\end{proposition}

Example \ref{example6} shows that our algorithm is non trivial, that is it shows a-monodromicity of arrangements for which other known results and algorithms cannot provide answers.

\section{Examples and Applications}
\label{ex}

In this section we will illustrate a couple of interesting examples to show how our algorithm can be applied to study monodromy of line arrangements. In particular we will also study the case of simplicial arrangements.

\begin{example}\label{example4}
Figure \ref{fig:example3} (respectively \ref{fig:example4}) corresponds to same sharp arrangement $\A$ with the choice of polar coordinate system $(V_0,V_1)$ as in Remark \ref{rem:diffpol} 1.i. (respectively Remark \ref{rem:diffpol} 1.ii.). This arrangement satisfies:
\begin{enumerate}
\item $m(P_0)=4$ divides $|\overline{\A}|=12;$ 
\item any line of $\overline{\A}$ contains at least two intersection points in $\PP^2_\R$ of multiplicity $4;$
\item any band of parallel lines in $\A$ is $4-$resonant (the band includes two unbounded chambers which are separated by $8$ hyperplanes), see \cite{Yoshi2,Yoshi3} for the definitions of band and $k-$resonance introduced by Yoshinaga;
\item the graph of double points $\Gamma(\A)$ in not connected, see $H_2^{P_3}$ in Figure \ref{fig:example3}.
\end{enumerate}
Since $m(P_0)$ and $m(P_3)$ are coprime and all the other points in $\mathcal{P}$ have multiplicity $2$ or $3,$ $$\A'=\{H_3^{P_0},H_2^{P_0},H_2^{P_1},H_3^{P_1}\} = \{\widetilde{H}_3^{P_0},\widetilde{H}_2^{P_0},\widetilde{H}_2^{P_5},\widetilde{H}_3^{P_5}\}.$$
In figure \ref{fig:example3}, $\last(H_2^{P_1})=\last(H_3^{P_0})$ and we consider $\mathcal{G}_{\A_0',\last,\mini},$ where $\A_0'=\A',$ which contains the following cycle:

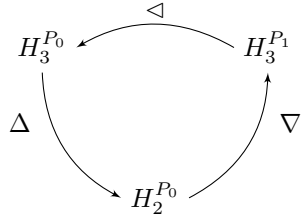
\begin{center}

\begin{tikzpicture}

\tikzset{vertex/.style = {shape=rectangle,minimum size=0.1em}}
\tikzset{edge/.style = {->,> = latex'}}
\node[vertex] (a) at  (0,0) {$H_3^{P_0}$};
\node[vertex] (b) at  (3,0) {$H_3^{P_1}$};
\node[vertex] (c) at  (1.5,-2) {$H_2^{P_0}$};

\draw[edge] (b.west)  to[bend right] (a.east);
\draw[edge] (a.south)  to[bend right] (c.west);
\draw[edge] (c.east)  to[bend right] (b.south);

\node[vertex] (d) at  (1.5,0.5) {$\pol$};
\node[vertex] (e) at  (-0.3,-1) {$\Delta$};
\node[vertex] (f) at  (3.3,-1) {$\nabla$};

\end{tikzpicture}
\end{center}

and no conclusion is possible. On the other hand, if we consider the polar coordinate system in Figure \ref{fig:example4}, the last points of the lines in $\A'$ are all different and the a-monodromicity of $\A$ follows directly from Theorem \ref{thm1.1}, since the only non trivial monodromy that can appear has order 3 coprime with $m(P_0).$

\end{example}

\begin{figure}[htbp]
\centering
\begin{tikzpicture}[scale=1]

\draw (0,0) node [right] {$H_3^{P_0}$} --(0,8);
\draw (1,0) node [right] {$H_2^{P_0}$} --(1,8);
\draw (2,0) node [right] {$H_1^{P_0}$} --(2,8);

\draw (2.5,2) node [right] {$H_4^{P_1}$} --(-3,2);
\draw (2.5,3) node [right] {$H_2^{P_2}$} --(-3,3);
\draw (2.5,4) node [right] {$H_2^{P_3}$} --(-3,4);

\draw (2.5,4.5) node [right] {$H_3^{P_3}$} --(-2,0);

\draw (2.5,5.5) node [right] {$H_2^{P_4}$} --(-2.5,0.5);

\draw (2.5,6.5) node [right] {$H_2^{P_5}$} --(-3,1);

\draw (2.5,1.5) node [right] {$H_3^{P_1}$} --(-3,7);
\draw (2.5,1) node [right] {$H_2^{P_1}$} --(-1,8);

\draw [-latex, red, dashed, rounded corners=80pt] (-0.5,0.5) node [left] {$V_0$} --(3.5,1)--(4,8) node [right] {$V_1$};

\end{tikzpicture}
\caption{$(V_0,V_1)$ as in Remark \ref{rem:diffpol} 1.i.}
\label{fig:example3}
\end{figure}
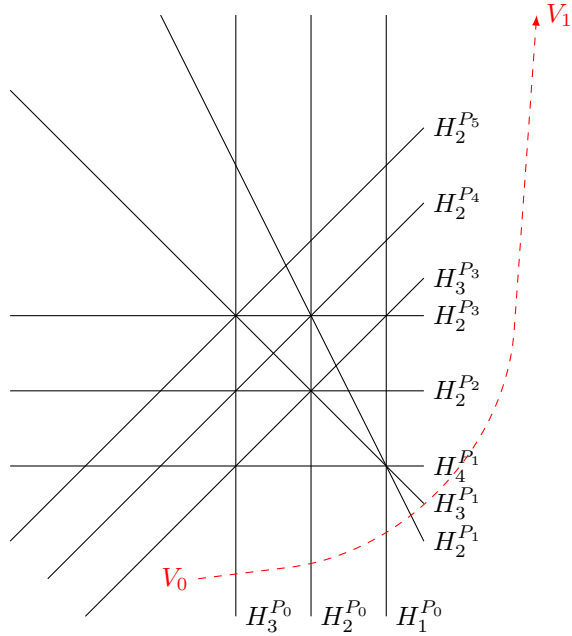

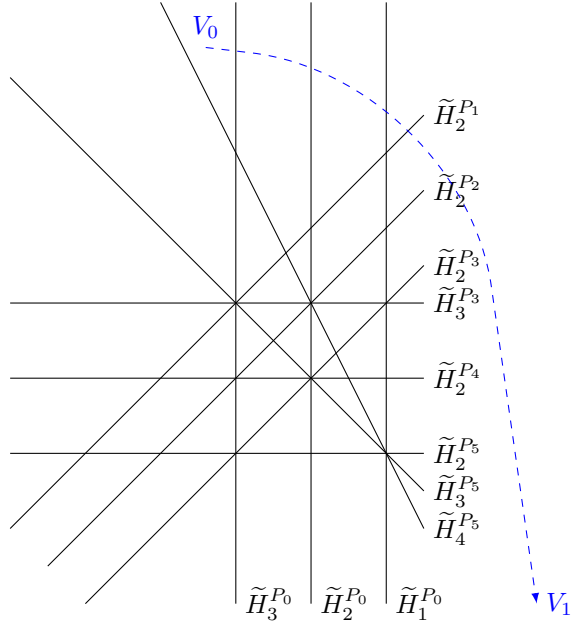
\begin{figure}[htbp]
\centering
\begin{tikzpicture}[scale=1]

\draw (0,0) node [right] {$\widetilde{H}_3^{P_0}$} --(0,8);
\draw (1,0) node [right] {$\widetilde{H}_2^{P_0}$} --(1,8);
\draw (2,0) node [right] {$\widetilde{H}_1^{P_0}$} --(2,8);

\draw (2.5,2) node [right] {$\widetilde{H}_2^{P_5}$} --(-3,2);
\draw (2.5,3) node [right] {$\widetilde{H}_2^{P_4}$} --(-3,3);
\draw (2.5,4) node [right] {$\widetilde{H}_3^{P_3}$} --(-3,4);

\draw (2.5,4.5) node [right] {$\widetilde{H}_2^{P_3}$} --(-2,0);

\draw (2.5,5.5) node [right] {$\widetilde{H}_2^{P_2}$} --(-2.5,0.5);

\draw (2.5,6.5) node [right] {$\widetilde{H}_2^{P_1}$} --(-3,1);

\draw (2.5,1.5) node [right] {$\widetilde{H}_3^{P_5}$} --(-3,7);
\draw (2.5,1) node [right] {$\widetilde{H}_4^{P_5}$} --(-1,8);

\draw [-latex, blue, dashed, rounded corners=80pt] (-0.4,7.4) node [above] {$V_0$} --(3,7)--(4,0) node [right] {$V_1$};

\end{tikzpicture}
\caption{$(V_0,V_1)$ as in Remark \ref{rem:diffpol} 1.ii.}
\label{fig:example4}
\end{figure}

\begin{remark}
It is also possible to prove a-monodromicity for the arrangement in Example \ref{example4} by using \cite[Theorem 3.23, Corollary 3.24]{Yoshi2}.
\end{remark}

\begin{example}[Simplicial arrangements]\label{ex:18lines}
An arrangement $\A$ in $\R^2$ is called \textit{simplicial} if each chamber of $\overline{\A}$ in $\PP_{\R}^2$ is a triangle. Gr\"{u}nbaum in \cite{Gru} presents a catalogue of known simplicial arrangements up to 37 lines (see \cite{Cuntz} for additional informations). In \cite{Yoshi2} Yoshinaga uses his algorithm to study the monodromy of the simplicial arrangement $\A(6m,1)$ obtained taking $3m$ lines determined by the sides of the $3m$-gon together with the $3m$ lines of symmetry of that $3m$-gon. He proved that it is $3$-monodromic (or \textit{pure-tone} using Yoshinaga definition). Yoshinaga also conjectured that those are the only simplicial arrangements with non trivial monodromy and that, more in general, if monodromy appears in a simplicial arrangements, it can only be equal to $3$. 
It is part of a work in progress to prove that if a simplicial arrangement contains three hyperplanes $H,H', H''$ such that $(H,H')$ and $(H',H'')$ are sharp pairs then the only non trivial monodromy that can appear is 3. In the following we give an example of how our algorithm reduces difficulty on computation to study a-monodromicity in simplicial arrangements. 

\begin{figure}[htbp]
\centering
\begin{tikzpicture}[scale=1]

\draw [rounded corners=20pt] (-5,5)--(5.5,-5.5)--(7.5,-6.5);
\draw [rounded corners=20pt] (-4,5)--(6,-5)--(7,-7);
\draw (6.55,-6.05) node {$\bullet$};
\draw (6.55,-6.05) node [above right]{$P_0$};
\draw (5.5,-4) node {$H_\infty$};
\draw (4,-4.5) node {$H_1^{P_0}$};

\draw [rounded corners=20pt] (3,5)--(-6,-4)--(-5.5,-7);
\draw [rounded corners=20pt] (4,5)--(-5.5,-4.5)--(-6,-7);
\draw (5,5)--(-6.5,-6.5);
\draw [rounded corners=20pt] (6,5)--(-4.5,-5.5)--(-7,-6);
\draw [rounded corners=20pt] (7,5)--(-4,-6)--(-7,-5.5);
\draw (0,-6.35) node {$\bullet$};

\draw [rounded corners=20pt] (-2,5)--(-2,-5)--(1,-7);
\draw [rounded corners=20pt] (-1,5)--(-1,-5)--(.5,-7);
\draw (0,-7)--(0,5);
\draw [rounded corners=20pt] (2,5)--(2,-5)--(-1,-7);
\draw [rounded corners=20pt] (1,5)--(1,-5)--(-.5,-7);
\draw (-5.72,-5.72) node {$\bullet$};

\draw [rounded corners=20pt] (6,-2)--(-6,-2)--(-8,1);
\draw (6,-2) node [right] {$H_2^{P_1}$};
\draw [rounded corners=20pt] (6,-1)--(-6,-1)--(-8,.5);
\draw (6,0)--(-8,0);
\draw [rounded corners=20pt] (6,1)--(-6,1)--(-8,-0.5);
\draw [rounded corners=20pt] (6,2)--(-6,2)--(-8,-1);
\draw (6,2) node [right] {$H_2^{P_7}$};
\draw (-7.3,0) node {$\bullet$};
\draw (-7.3,0) node [above] {\footnotesize{$\last(H_2^{P_0})$}};

\draw [dashed, rounded corners=40pt] (8,-6)--(-3.5,-6.5)--(-5.72,-5.72);

\draw [dashed, rounded corners=40pt] (-5.72,-5.72)--(-7,-4)--(-7.6,4);

\draw (2,-2) node {$\bullet$};
\draw (2,-2) node [below left]{$P_1$};

\draw (1,-1) node {$\bullet$};
\draw (1,-1) node [below left]{$P_2$};

\draw (.5,-.5) node {$\bullet$};
\draw (.5,-.5) node [below left]{$P_3$};

\draw (0,0) node {$\bullet$};
\draw (0,0) node [below left]{$P_4$};

\draw (-0.5,0.5) node {$\bullet$};
\draw (-.5,.5) node [below left]{$P_5$};

\draw (-1,1) node {$\bullet$};
\draw (-1,1) node [below left]{$P_6$};

\draw (-2,2) node {$\bullet$};
\draw (-2,2) node [below left]{$P_7$};

\draw (4,-5.9) node {$H_2^{P_0}$};

\draw (0,2) node {$\bullet$};
\draw (0,2) node [above left] {\footnotesize{$\last(H_2^{P_7})$}};

\draw (1,-2) node {$\bullet$};
\draw (1,-2) node [above left] {\footnotesize{$\mini(H_2^{P_1})$}};

\draw (-0.3,-6.35) node [left] {$\widetilde{P}$};

\draw [red] (5.55,-7.2) node {$\bullet$};
\draw [-latex, red, dashed, rounded corners=80pt] (5.5,-7) node [left] {$V_0$} --(5,-4.5)--(-3,3.5) node [above left] {$V_1$};

\end{tikzpicture}
\caption{Simplicial arrangement $\overline{\A}$ in $\PP_{\R}^2$ with 18 lines}
\label{fig:18lines}
\end{figure}

In the simplicial arrangement depicted in Figure \ref{fig:18lines} (known to be a-monodromic), $(H_1^{P_0},H_\infty)$ is a sharp pair of lines, since there is no intersection points contained between them. With this choice of sharp pair we get that multiplicities $m(P_0)=3$ and $m(P_2)=m(P_4)=m(P_6)=4$ are coprime and $P_3$ and $P_5$ are double points. Hence the set $\A_0=\A'_0$ defined in Notation \ref{not:arrnozero} is $$\A_0=\A'_0 = \{H_2^{P_0},H_2^{P_1},H_2^{P_7}\}$$ 
and the study of the boundary matrix simply reduces to the study of a three rows matrix containg the three columns triangular submatrix ( see Notation \ref{not:goodcolumn} ) 
$$\bordermatrix{
&c_{H_2^{P_0}}^{\last(H_2^{P_0})}&c_{H_2^{P_1}}^{\mini(H_2^{P_1})}&c_{H_2^{P_7}}^{\last(H_2^{P_7})}\cr
& & & \cr
H_2^{P_0}&t^{\alpha_0}(1-t)&0 & 0\cr
H_2^{P_1}& * & t^{\alpha_1}(1-t)& 0\cr
H_2^{P_7}& * & 0 & t^{\alpha_7}(1-t)\cr}$$
that is $\A$ is a-monodromic. 
\end{example}


\begin{thebibliography}{99}
\bibitem{Bailet} P. Bailet: On the Monodromy of Milnor Fibers of Hyperplane Arrangements, Canadian Mathematical Bulletin 57(2014), no. 4, 697-707.

\bibitem{Bailet2} P. Bailet, M. Yoshinaga: Degeneration of Orlik-Solomon algebras and Milnor fibers of complex line arrangements. Geometriae Dedicata, vol. 175 (2015) 49-56.


\bibitem{BDS} N. Budur, A. Dimca and M. Saito: First Milnor cohomology of hyperplane arrangements. 'Topology of Algebraic Varieties and Singularities', Contemporary Mathematics 538(2011), 279-292.

\bibitem{BS} N. Budur, M. Saito:  Jumping coefficients and spectrum of a hyperplane arrangement. Math. Ann. 347 (2010), no. 3, 545–579.

\bibitem{BS1}  N. Budur, M. Saito: Multiplier ideals, V-filtration, and spectrum. J. Algebraic Geom. 14 (2005), no. 2, 269–282.

\bibitem{cdo}
D. Cohen, A. Dimca, P. Orlik: 
Nonresonance conditions for arrangements. \emph{Ann. Inst. Fourier} {\bf 53} (2003), 1883--1896. 

\bibitem{co-loc}
D. Cohen, P. Orlik: 
Arrangements and local systems. \emph{Math. Res. Lett.} \textbf{7} (2000), 299-316.

\bibitem{Cuntz}M. Cuntz: Simplicial arrangements with up to 27 lines; \textit{Discrete Comput. Geom. } \textbf{48} (2012), 682--701.


\bibitem{Dimcamonodromy} A. Dimca: Monodromy of triple point line arrangements, Singularities in Geometry and Topology 2011,
Eds. V. Blanloeil, O. Saeki, Adv. Studies in Pure Math. 66 (2015), pp. 71-80, Math. Society of Japan.

\bibitem{ST} A. Dimca: Singularities and topology of hypersurfaces, Universitext, Springer-Verlag, New York, 1992.

\bibitem{DNA} A. Dimca: Tate properties, polynomial-count varieties, and monodromy of hyperplane arrangements. (English summary) Nagoya Math. J. 206 (2012), 75–97.
 

\bibitem{DL}  A. Dimca, G. Lehrer: Hodge-Deligne equivariant polynomials and monodromy of hyperplane arrangements. Configuration Spaces CRM Series 2012, pp 231-253. 




\bibitem{DP} A. Dimca, S. Papadima: Finite Galois covers, cohomology jump loci, formality properties, and multinets. (English summary)
Ann. Sc. Norm. Super. Pisa Cl. Sci. (5) 10 (2011), no. 2, 253–268. 



\bibitem{D-P} A. Dimca, S. Papadima: Hypersurface complements, Milnor fibers
and higher homotopy groups of arrangements; Ann. of Math. (2) 158
(2003), no.2, 473--507
 

\bibitem{D3}A. Dimca, G. Sticlaru: A computational approach to Milnor fiber cohomology, arXiv:1602.03496.


\bibitem{FalYuz}M. Falk, S. Yuzvinsky: Multinets, resonance varieties, and pencils of plane curves; \textit{Compos. Math.} 143 (2007), no. 4, 1069--1088.

\bibitem{Gru}B. Gr\"{u}nbaum: A catalogue of simplicial arrangements in the real projective plane; \textit{Ars Math. Contemp.} 2 (2009), no.1, 1--25.

\bibitem{gaiffi2009morse} G. Gaiffi, M. Salvetti: The morse complex of a line arrangement; \textit{Journal of algebra}, 321(1):316-337, 2009.

\bibitem{lib-mil}
A. Libgober: Eigenvalues for the monodromy of the Milnor fibers of arrangements. 
{\em Trends in singularities}, 141-150, Trends Math., Birkh\"auser, 
Basel, 2002. 


\bibitem{PM} A. M\u acinic, S. Papadima:
On the monodromy action on Milnor fibers of graphic arrangements. Topology Appl. 156 (2009), no. 4, 761–-774.


\bibitem{O-S} P. Orlik, L. Solomon: Combinatorics and topology of complements onf hyperplanes. Invent. Math.56(1980) 167-189.

\bibitem{O-T}  P. Orlik, H. Terao: Arrangements of hyperplanes, Grundlehren Math. Wiss. 300, SpringerVerlag, Berlin, 1992. 


\bibitem{PS}  S. Papadima, A. Suciu: The spectral sequence of an equivariant chain complex and homology with local coefficients. Trans. Amer. Math. Soc. 362 (2010), 2685-2721.

\bibitem{PS2}S. Papadima, A. Suciu: The Milnor fibration of a hyperplane arrangement: from modular resonance to algebraic monodromy, arXiv:1401.0868.

\bibitem{Ran} R. Randell: Morse theory, Milnor fibers and minimality of a complex hyperplane arrangement, Proc. Amer. Math. Soc. 130 (2002), no. 9, 2737–2743.

\bibitem{Sal} M. Salvetti: Topology of the complement of real hyperplanes in $\C^n$,
Inv. Math., 88 (1987), no.3, 603–618.

\bibitem{Sal-Ser} M. Salvetti, M. Serventi: On the twisted cohomology of arrangements of lines, arXiv:1507.01670v1

\bibitem{SS} M. Salvetti, S. Settepanella: Combinatorial Morse theory and minimality of hyperplane arrangements; Geom. Topol., 11:1733–1766, 2007.

\bibitem{S1} S. Settepanella:
Cohomology of pure braid groups of exceptional cases. 
Topology Appl. 156 (2009), no. 5, 1008–1012. 

\bibitem{Torielli-Yoshi} M. Torielli, M. Yoshinaga: Resonant bands, Aomoto complex, and real 4-nets. Journal of Singularities, vol. 11 (2015), 33-51.

\bibitem{Wil}K. Williams: The homology groups of the Milnor fiber associated to a central arrangement of hyperplanes in $\C^3.$ Topol. Appl. \textbf{160}(10), 1129-1143 (2013).





\bibitem{y-lef} 
M. Yoshinaga: 
Hyperplane arrangements and Lefschetz's hyperplane section theorem. 
\emph{Kodai Math. J.} \textbf{30}, no. 2 (2007), 157--194.

\bibitem{Yoshi2}M. Yoshinaga:
Milnor fibers of real line arrangements. 
\emph{Journal of Singularities}, \textbf{7} (2013), 220-237.

\bibitem{y-mini} 
M. Yoshinaga: Minimality of hyperplane arrangements and basis of local system cohomology. \emph{Singularities in geometry and topology}, 345-362, IRMA Lect. Math. Theor. Phys., 20, \emph{Eur. Math. Soc., Z\"urich}, 2012.

\bibitem{Yoshi3}
M. Yoshinaga: Resonant bands and local system cohomology groups for real line arrangements. Vietnam J. Math. 42 (2014), no. 3, 377-392.

\bibitem{y-cham} M. Yoshinaga: The chamber basis of the Orlik-Solomon algebra and Aomoto complex. \emph{Ark. Mat.} \textbf{47} (2009), no. 2, 393-407. 



\end{thebibliography}
\end{document}